\documentclass[11pt,reqno]{amsart}
\usepackage{mlmodern}

\usepackage[shortlabels]{enumitem}

\usepackage[nosumlimits]{mathtools}
\usepackage{amssymb}

\usepackage{graphicx}
\usepackage{subcaption}
\usepackage{algorithm}
\usepackage{algpseudocode}

\usepackage{eucal}
\usepackage[scr]{rsfso}

\usepackage{tikz-cd}
\usepackage[margin=1in]{geometry}
\usepackage{hyperref,xcolor}
\hypersetup{
    colorlinks,
    linkcolor={red!50!black},
    citecolor={blue!50!black},
    urlcolor={blue!80!black}
}
\usepackage{romannum}
\AtBeginDocument{\pagenumbering{arabic}}

\theoremstyle{plain}
\newtheorem{theorem}{Theorem}[section]
\newtheorem{proposition}[theorem]{Proposition}
\newtheorem{lemma}[theorem]{Lemma}
\newtheorem{corollary}[theorem]{Corollary}
\theoremstyle{definition}

\newtheorem{example}[theorem]{Example}

\newcommand{\F}{{\scriptscriptstyle\mathsf{F}}}
\newcommand{\tp}{{\scriptscriptstyle\mathsf{T}}}
\newcommand{\ha}{{\scriptscriptstyle\mathsf{H}}}
\newcommand{\p}{{\scriptscriptstyle+}}
\newcommand{\m}{{\scriptscriptstyle-}}

\let\O\undefine
\DeclareMathOperator{\O}{O}
\DeclareMathOperator{\U}{U}
\DeclareMathOperator{\GL}{GL}
\DeclareMathOperator{\diag}{diag}
\DeclareMathOperator{\tr}{tr}
\DeclareMathOperator{\spn}{span}
\DeclareMathOperator{\op}{\mathsf{op}}
\DeclareMathOperator{\add}{\mathsf{add}}
\DeclareMathOperator{\inv}{\mathsf{inv}}
\DeclareMathOperator{\mul}{\mathsf{mul}}

\DeclareMathOperator{\res}{res}
\DeclareMathOperator{\sign}{sign}

\DeclareMathOperator{\GF}{GF}
\DeclareMathOperator{\num}{\texttt{\#}}

\begin{document}
\title{Complex matrix inversion via real matrix inversions}
\author[Z.~Dai]{Zhen~Dai}
\author[L.-H.~Lim]{Lek-Heng~Lim}
\address{Computational and Applied Mathematics Initiative, University of Chicago, Chicago, IL 60637-1514}
\email{zhen9@uchicago.edu, lekheng@uchicago.edu}
\author[K.~Ye]{Ke Ye}
\address{KLMM, Academy of Mathematics and Systems Science, Chinese Academy of Sciences, Beijing 100190, China}
\email{keyk@amss.ac.cn}

\date{}

\begin{abstract}
We study the inversion analog of the well-known Gauss algorithm for multiplying complex matrices. A simple version is $(A + iB)^{-1} = (A + BA^{-1}B)^{-1} - i A^{-1}B(A+BA^{-1} B)^{-1}$ when $A$ is invertible, which may be traced back to Frobenius but has received scant attention. We prove that it is optimal, requiring fewest matrix multiplications and inversions over the base field, and we extend it in three ways: (i) to any invertible $A + iB$ without requiring $A$ or $B$ be invertible; (ii) to any iterated quadratic extension fields, with $\mathbb{C}$ over $\mathbb{R}$ a special case; (iii) to Hermitian positive definite matrices $A + iB$ by exploiting symmetric positive definiteness of $A$ and $A + BA^{-1}B$. We call all such algorithms Frobenius inversions, which we will see do not follow from Sherman--Morrison--Woodbury type identities and cannot be extended to Moore--Penrose pseudoinverse. We show that a complex matrix with well-conditioned real and imaginary parts can be arbitrarily ill-conditioned, a situation tailor-made for Frobenius inversion. We prove that Frobenius inversion for complex matrices is faster than standard inversion by LU decomposition and Frobenius inversion for Hermitian positive definite matrices is faster than standard inversion by Cholesky decomposition.  We provide extensive numerical experiments, applying Frobenius inversion to solve linear systems, evaluate matrix sign function, solve Sylvester equation, and compute polar decomposition, showing that Frobenius inversion can be more efficient than LU/Cholesky decomposition with negligible loss in accuracy.  A side result is a generalization of Gauss multiplication to iterated quadratic extensions, which we show is intimately related to the Karatsuba algorithm for fast integer multiplication and multidimensional fast Fourier transform. 
\end{abstract}


\maketitle

\section{Introduction}\label{sec:intro}

The article is a sequel to our recent work in \cite{Zhen}, where we studied the celebrated Gauss multiplication algorithm $(A + iB)(C+ iD) =  (AC - BD) + i[(A+B)(C+D) -AC -BD]$ for multiplying a pair of complex matrices with just three real matrix multiplications. Such methods for performing a complex matrix operation in terms of real matrix operations can be very useful as floating point standards such as the IEEE-754  \cite{ieee} often do not implement complex arithmetic natively but rely on software to reduce complex arithmetic to real arithmetic \cite[p.~55]{Overton}. Here we will analyze and extend an inversion analogue of Gauss algorithm: Given a complex invertible matrix $A+i B \in \mathbb{C}^{n\times n}$ with $A,B \in \mathbb{R}^{n \times n}$, it is straightforward to verify that its inverse is given by
\begin{equation}\label{eq:Finv}
(A + i B)^{-1} = 
(A + BA^{-1}B)^{-1} - i A^{-1}B(A+BA^{-1} B)^{-1}
\end{equation}
if $A$ is invertible, a formula that can be traced back to Georg Frobenius  \cite{inversion_ref8}. In our article we will refer to all such algorithms and their variants and extensions as \emph{Frobenius inversions}.  While Gauss multiplication has been thoroughly studied (two representative references are \cite[Section 4.6.4]{KnuthBook} in Computer Science and \cite[Section~23.2.4]{Higham} in Numerical Analysis, with numerous additional references therein), the same cannot be said of Frobenius inversion --- we combed through the research literature and found only six references,  all from the 1970s or earlier, which we will review in Section~\ref{sec:prev}.

Our goal is to vastly extend and thoroughly analyze Frobenius inversion from a modern perspective.  We will extend it to the general case where only $A + iB$ is invertible but neither $A$ nor $B$ is (Section~\ref{sec:rand}), and to the important special case where $A + iB$ is Hermitian positive definite, in a way that exploits the symmetric positive definiteness of $A$ and  $A + BA^{-1}B$ (Section~\ref{sec:herm}).  We will show (Section~\ref{sec:lin}) that it is easy to find complex matrices $A + iB$ with
\begin{equation}\label{eq:cond}
\max\bigl( \kappa_2(A), \kappa_2(B), \kappa_2 (A + BA^{-1}B) \bigr) \ll \kappa_2(A + iB),
\end{equation}
where the gap between the left- and right-hand side is arbitrarily large, i.e., $A + iB$ can be arbitrarily ill-conditioned even when $A, B, A + BA^{-1}B$ are all well-conditioned --- a scenario bespoke for \eqref{eq:Finv}.

Frobenius inversion obviously extends to any quadratic fields of the form $\Bbbk[\sqrt{a}]$, i.e., $x^2 + a$ is irreducible over $\Bbbk$, but we will further extend it to any arbitrary quadratic field, and any iterated quadratic extensions including constructible numbers, multiquadratics, and towers of root extensions (Section~\ref{sec:tower}). In fact we show that for iterated quadratic extensions, Frobenius inversion essentially gives the multidimensional fast Fourier transform. We will prove that over any quadratic field Frobenius inversion is optimal in that it requires the least number of matrix multiplications and inversions over its base field (Sections~\ref{sec:Frob inv}, \ref{sec:herm}, and \ref{sec:rand}).

For complex matrix inversion, we show that  \textsc{Matlab}'s built-in inversion algorithm, i.e., directly inverting a matrix with LU or Cholesky decomposition \emph{in complex arithmetic}, is slower than applying Frobenius inversion with LU or Cholesky decomposition \emph{in real arithmetic} (Theorem~\ref{thm:threshold}, Propositions~\ref{prop:threshold} and \ref{prop:threshold2}). More importantly, we provide a series of numerical experiments in Section~\ref{sec:expr} to show that Frobenius inversion is indeed faster than \textsc{Matlab}'s built-in inversion algorithm in almost every situation and, despite well-known exhortations to avoid matrix inversion, suffers from no significant loss in accuracy. In fact methods based on Frobenius inversion may  be more accurate than standard methods in certain scenarios (Section~\ref{sec:linexp}).

\subsection{Why not invert matrices}\label{sec:no}

Matrix inversion is frown upon in numerical linear algebra, likely an important cause for the lack of interest in algorithms like Frobenius inversion. The usual reason for eschewing inversion \cite{HighamBook} is that in solving an $n \times n$ nonsingular system $Ax = b$, if we compute a solution $\widehat{x}_{\inv}$ by inverting $A$ through LU factorization $PA = LU$ and multiplying $A^{-1}$  to $b$, and if we compute a solution $\widehat{x}_{LU}$ directly through the LU factors with backward substitutions $Ly = Pb$, $Ux = y$, the latter approach is both faster, with $2n^3$ flops for $\widehat{x}_{\inv}$ versus $2n^3/3$ for $\widehat{x}_{LU}$, and more accurate, with backward errors
\begin{equation}\label{eq:compare}
   \lvert b - A\widehat{x}_{\inv} \rvert \leq n \lvert A\rvert \lvert A^{-1}\rvert \lvert b\rvert \mathsf{u} + O(\mathsf{u}^2) \quad\text{versus}\quad
   \lvert b - A\widehat{x}_{LU} \rvert \leq 3n \lvert \widehat{L}\rvert \lvert \widehat{U}\rvert \lvert \widehat{x}_{LU}\rvert \mathsf{u} + O(\mathsf{u}^2).
\end{equation}
Here $\mathsf{u}$ denotes  unit roundoff and  where $|\cdot|$ and $\le$ applies componentwise. As noted in  \cite{HighamBook}, usually $\lVert |\widehat{L}||\widehat{U}| \rVert_{\infty} \approx \lVert A \rVert_{\infty}$ and so $\widehat{x}_{LU}$ is likely more accurate than $\widehat{x}_{\inv}$ when $\lVert x \rVert_{\infty} \ll \lVert |A^{-1}||b| \rVert_{\infty}$.

Another common rationale for avoiding inversion is the old wisdom that many tasks that appear to require inversion actually do not --- an explicit inverse matrix $A^{-1} \in \mathbb{C}^{n \times n}$ is almost never required because upon careful examination, one would invariably realize that the same objective could be accomplished with a vector like $A^{-1}b$ or $\diag(A^{-1}) \in \mathbb{C}^n$ or a scalar like $c^\tp A^{-1} b$, $\lVert A^{-1} \rVert$, or $\tr(A^{-1}) \in \mathbb{C}$. These vectors and scalars could be computed with a matrix factorization or approximated to arbitrary accuracy with iterative methods \cite{golub}, which are often more amenable to updating/downdating \cite{GGMS} or better suited for preserving structures like sparsity.

\subsection*{Caveat} We  emphasize that Frobenius inversion, when applied to solve a system of complex linear equations $(A + iB)z = c + id$, will \emph{not} involve actually computing an explicit inverse matrix $(A + iB)^{-1}$ and then multiplying it to the vector  $c + id$. In other words, we do not use the expression in \eqref{eq:Finv} literally but only apply it in conjunction with various LU decompositions and back substitutions over $\mathbb{R}$; the matrix $(A + iB)^{-1}$ is never explicitly formed. The details are given in Section~\ref{sec:lin} alongside discussions of circumstances like \eqref{eq:cond} where the use of Frobenius inversion gives more accurate results than standard methods, with numerical evidence in Section~\ref{sec:linexp}.

\subsection{Why invert matrices}\label{sec:yes}

 We do not dispute the reasons in Section~\ref{sec:no} but numerical linear algebra is a field that benefits from a wide variety of different methods for the same task, each suitable for a different regime. There is no single method that is universally best in every instance. Even the normal equation, frown upon in numerical linear algebra like matrix inversion, can be the ideal method for certain least squares problems.

In fact, if we examine the  advantages of computing $\widehat{x}_{LU}$ over  $\widehat{x}_{\inv}$ in Section~\ref{sec:no}  more closely, we will find that the conclusion is not so clear cut. Firstly, the comparison in \eqref{eq:compare} assumes that accuracy is quantified by backward error $\lvert b - A\widehat{x} \rvert $ but in reality it is the forward error $\lvert x - \widehat{x} \rvert $ that is far more important and investigations in \cite{seqlinear,invforwarderr}, both analytical and experimental, show that the forward errors of $\widehat{x}_{LU}$ and  $\widehat{x}_{\inv}$ are similar. Secondly, if instead of solving a single linear system $Ax = b$, we have $p$ right-hand sides $b_1,\dots,b_p\in \mathbb{C}^n$, then it becomes $AX = B$ where $B =[b_1,\dots,b_p] \in \mathbb{C}^{n \times p}$ and we seek a solution $X \in \mathbb{C}^{n \times p}$. In this case the speed advantage of computing $\widehat{X}_{LU}$ over $\widehat{X}_{\inv}$  disappears when  $p =O(n)$: Note that the earlier flop count $2n^3/3$ for $\widehat{x}_{LU}$ ignores the cost of two backsubstitutions but when there are $2p$ backsubstitutions, these may no longer be ignored and are in fact dominant, making the cost of computing $\widehat{X}_{\inv}$ and $\widehat{X}_{LU}$ comparable. In \cite{seqlinear}, it is shown that because of data structure complications, computing $\widehat{X}_{\inv}$ can be significantly faster than  $\widehat{X}_{LU}$.

Moreover, the old wisdom that one may avoid computing explicit inverse matrices, while largely true, is not always true. There are situations, some of them alluded to in \cite[p.~260]{HighamBook}, where computing an explicit inverse matrix is inevitable or favorable:
\begin{description}[font=\normalfont, leftmargin=1.5\parindent]
\item[\textsc{Mimo radios}] In such radios, explicit inverse matrices are implemented  in hardware \cite{invforwarderr,hardware1,hardware2}. It is straightforward to hardwire or hardcode an explicit inverse matrix but considerably more difficult to do so in the form of ``LU factors with permutations and backsubstitutions,'' which can require more gates or code space and is more prone to implementation errors.

\item[\textsc{Superconductivity}] In the so-called KKR CPA algorithm \cite{heath1991early}, one needs to integrate the KKR inverse matrix over the first Brillouin zone, necessitating an explicit inverse matrix.

\item[\textsc{Linear modeling}] The inverse of a matrix often reveals important statistical properties that could only be discerned when one has access to the full explicit inverse \cite{maindonald1984statistical, mccullagh1989generalized}, i.e., we do not know which entries of $A^{-1}$ matter until we see all of them. For a specific example, take the ubiquitous model $y = X \widehat{\beta} + \varepsilon$ with design matrix $X$ and observed values $y_1,\dots, y_n$ of the dependent variable $y$  \cite{mccullagh1989generalized}, we understand the regression  coefficients $\widehat{\beta}$ through the values its covariance matrix $\Sigma \coloneqq \sigma^2 \cdot (X^\tp X)^{-1}$ where $\sigma^2$ is the variance of the dependent variable \cite{mccullagh1989generalized}. To see which values are large (positively correlated), small (negatively correlated), or nearly zero (uncorrelated) in relation to other values, we need access to all values of $\Sigma$.

\item[\textsc{Statistics}] For an unbiased estimator $\widehat{\theta}(X)$ of a parameter $\theta$, its Cramer--Rao lower bound is the inverse of its Fisher information matrix $I(\theta)$. This is an important quantity that gives a lower bound for the covariance matrix \cite{cramer_rao_1, cramer_rao_2} in the sense of  $\operatorname{cov}_{\theta}\bigl(\widehat{\theta}(X)\bigr)
\succeq  I(\theta)^{-1}$ where $\succeq$ is the Loewner order. In some Gaussian processes, this lower bound could be attained \cite{cramer_rao}. We need the explicit matrix inverse $I(\theta)^{-1}$ to understand the limits of certain statistical problems and to design optimal estimators that attain the Cramer--Rao lower bound.

\item[\textsc{Graph theory}] The inverses of the adjacency matrix, forward adjacency matrix, and various graph Laplacians of a graph $G$ contain important combinatorial properties about $G$ \cite{graph_ref1,graph_ref4,graph_ref2,graph_ref3} that are only revealed when one examines all entries of their explicit inverse matrices.

\item[\textsc{Symbolic computing}] Matrix inversions do not just arise in numerical computing with floating point operations. They are routinely performed in finite field arithmetic over a base field of the form $\Bbbk = \GF(p^n)$ in cryptography \cite{Hoffstein,Stinson}, combinatorics \cite{Krattenthaler96}, information theory \cite{AL69}, and finite field matrix computations \cite{SR22}. They are also carried out in rational arithmetic over transcendental fields \cite{Eberly97,EGGSV07,Giesbrecht95}, e.g., with a base field of the form $\Bbbk = \mathbb{Q}(x_1,\dots,x_n, e^{x_1},\dots,e^{x_n})$ and an extension field of the form $\mathbb{F} = \mathbb{Q}[i](x_1,\dots,x_n, e^{x_1},\dots,e^{x_n})$, or with finite fields in place of $\mathbb{Q}$ and $\mathbb{Q}[i]$. With such exact arithmetic, the considerations in Section~\ref{sec:no} become irrelevant.
\end{description}
In summary, the Frobenius inversion algorithms in this article are useful  (i) for problems with well-conditioned $A$, $B$, and $A + BA^{-1}B$ but  ill-conditioned $A + iB$; (ii) in situations requiring an explicit inverse matrix; (iii) to applications involving exact  finite field or rational arithmetic.

\subsection{Previous works}\label{sec:prev}

We review existing works that mentioned the inversion formula \eqref{eq:Finv} in the research literature: \cite{inversion_ref3,1100887,inversion_ref6,inversion_ref7,inversion_ref8,inversion_ref9} --- we note that this is an exhaustive list, and all predate 1979. We also widened our search to books and the education literature, and found \cite{inversion_ref2,inversion_ref5} in engineering education publications,  \cite[pp.~218--219]{inversion_ref1}, \cite[Exercise~14.8]{Higham}, and \cite[Chapter~II, Section~20]{inversion_ref4}, although they contain no new material. 

The algorithm,  according to \cite{inversion_ref8},  was first discovered by Frobenius and Schur although we are unable to find a published record in their Collected Works \cite{Frob, Schur}.  Since  ``Schur inversion'' is already used to mean something unconnected to complex matrices, and calling \eqref{eq:Finv} ``Frobenius--Schur inversion'' might lead to unintended confusion with Schur inversion, it seems befitting to name \eqref{eq:Finv} after Frobenius alone.

The discussions in \cite{inversion_ref3,inversion_ref6,inversion_ref9} are all about deriving Frobenius inversion.
From a modern perspective, the key to these derivations is an embedding of $\mathbb{C}^{n \times n}$ into $\mathbb{R}^{2n \times 2n}$ as a subalgebra via
\[
A + iB \mapsto \begin{bmatrix}
A & -B \\
B & A
\end{bmatrix}  \eqqcolon M,
\]
and noting that if $A$ is invertible, then $(A+iB)^{-1}$ corresponds to $M^{-1}$, given by the standard expression
\[
M^{-1} = \begin{bmatrix}
A^{-1} - A^{-1}B (M/A)^{-1} BA^{-1} & A^{-1}B(M/A)^{-1} \\ 
-(M/A)^{-1} B A^{-1} & (M/A)^{-1}
\end{bmatrix},
\]
where $M/A \coloneqq A + BA^{-1}B$ denotes the Schur complement of $A$ in $M$. The two right blocks of $M^{-1}$ then yield the expression 
\[
(A + iB)^{-1} = (M/A)^{-1} - i A^{-1}B(M/A)^{-1},
\]
which is  \eqref{eq:Finv}. The works in \cite{inversion_ref4,inversion_ref7} go further in addressing the case when both $A$ and $B$ are singular  \cite{inversion_ref4} and the case when $A$, $B$, $A+B$ or $A-B$ are all singular \cite{inversion_ref7}. However, they require the inversion of a $2n \times 2n $ real matrix, wiping out any computational savings that Frobenius inversion affords. The works \cite{1100887,inversion_ref9} avoided this pitfall but still compromised the computational savings of Frobenius inversion.  Our method in Section~\ref{sec:rand} will cover these cases and more, all while preserving the computational complexity of Frobenius inversion.

\subsection{Notations and conventions}

Fields are denoted in blackboard bold fonts. We write
\begin{align*}
\GL_n(\mathbb{F}) &\coloneqq \{X \in \mathbb{F}^{n \times n}: \det(X) \ne 0\}, \\
\O_n(\mathbb{R}) &\coloneqq \{X \in  \mathbb{R}^{n \times n} : X^\tp X = I \}, \\
\U_n(\mathbb{C}) &\coloneqq \{X \in  \mathbb{C}^{n \times n} : X^\ha  X = I \}
\end{align*}
for the general linear group of invertible matrices over any field $\mathbb{F}$, the orthogonal group over $\mathbb{R}$, and the unitary group over $\mathbb{C}$ respectively. Note that we have written $X^\tp$ for the transpose and $X^\ha$ for conjugate transpose for any $X \in \mathbb{C}^{m \times n}$. Clearly,  $X^\ha = X^\tp$ if $X \in \mathbb{R}^{m \times n}$.  We will also adopt the convention that $X^{-\tp} \coloneqq (X^{-1})^\tp = (X^\tp)^{-1}$ and $X^{-\ha} \coloneqq (X^{-1})^\ha = (X^\ha)^{-1}$ for any $X \in \GL_n(\mathbb{C})$. Clearly, $X^{-\ha} = X^{-\tp}$ if $X \in \GL_n(\mathbb{R})$.

For $\mathbb{F} = \mathbb{R}$ or $\mathbb{C}$, we write $\lVert X \rVert \coloneqq \sigma_1(X)$ for the spectral norm of $X \in \mathbb{F}^{m \times n}$ and $\kappa (X) \coloneqq \sigma_1(X)/\sigma_n(X)$ for the spectral condition number of $X \in \GL_n(\mathbb{F})$. When we speak of norm or condition number in this article, it will always be the spectral norm or spectral condition number, the only exception is the max norm defined and used in Section~\ref{sec:expr}.

\section{Frobenius inversion in exact arithmetic}\label{sec:symb}

We will first show that Frobenius inversion works over any quadratic field extension, with $\mathbb{C}$ over $\mathbb{R}$ a special case. More importantly, we will show that  Frobenius inversion is optimal over any quadratic field extension in that it requires a minimal number of matrix multiplications, inversions, and additions (Theorem~\ref{thm:optimal inverse}).

The reason for the generality in this section is to show that Frobenius inversion can be useful beyond numerical analysis, applying to matrix inversions in computational number theory \cite{Cohen1,Cohen2}, computer algebra \cite{Mignotte,Mishra}, cryptography \cite{Hoffstein,Stinson}, and finite fields \cite{McEliece,Menezes}  as well. This section covers the symbolic computing aspects of Frobenius inversion, i.e., in exact arithmetic. Issues related to  the numerical computing aspects, i.e., in floating-point arithmetic, including conditioning, positive definiteness, etc, will be treated in Sections~\ref{sec:lin}--\ref{sec:expr}.

Recall that a field  $\mathbb{F}$ is said to be a \emph{field extension} of another field $\Bbbk$ if $\Bbbk \subseteq \mathbb{F}$. In this case, $\mathbb{F}$ is automatically   a $\Bbbk$-vector space. The dimension of $\mathbb{F}$ as a $\Bbbk$-vector space is called the \emph{degree} of $\mathbb{F}$ over $\Bbbk$ and denoted   $[\mathbb{F} : \Bbbk]$ \cite{Roman}. A degree-two extension is also called a \emph{quadratic extension} and they are among the most important field extensions. For example, in number theory, two of the biggest achievements in the last decade were the generalizations of Andrew Wiles' celebrated work to real quadratic fields \cite{modularityR} and imaginary quadratic fields \cite{modularityI}. Let $\mathbb{F}$ be a quadratic extension of $\Bbbk$. Then it follows from standard field theory \cite{Roman} that there exists some monic irreducible quadratic polynomial $f \in \Bbbk[x]$ such that 
\[
\mathbb{F} \simeq \Bbbk[x] \!\bigm/\!\! \langle f\rangle,
\]
where $\langle f\rangle$ denotes the principal ideal generated by $f$ and $\Bbbk[x]/\langle f\rangle$ the quotient ring.
Let $f(x) = x^2 + \beta x + \tau$ for some $\beta,\tau \in \Bbbk$. Then, up to an isomorphism, $f$ may be written in a normal form:   
\begin{itemize}
\item $\operatorname{char}(\Bbbk) \ne 2$: $\beta = 0$ and $-\tau$ is not a complete square in $\Bbbk$;
\item $\operatorname{char}(\Bbbk) = 2$: either $\beta = 0$ and $-\tau$ is not a complete square in $\Bbbk$, or $\beta = 1$ and $x^2 + x + \tau$ has no solution in $\Bbbk$. 
\end{itemize}

\subsection{Gauss multiplication over quadratic field extensions}\label{sec:gauss}

Let $\xi$ be a root of $f(x)$ in an algebraic closure $\overline{\Bbbk}$. Then $\mathbb{F} \simeq \Bbbk[\xi]$, i.e., any element in $\mathbb{F}$ can be written uniquely as $a_1 + a_2 \xi$ with $a_1,a_2 \in \Bbbk$. Henceforth we will assume that $\mathbb{F} = \Bbbk[\xi]$. The product of two elements $a_1 + a_2 \xi$, $b_1 + b_2 \xi \in \Bbbk[\xi]$ is given by
\begin{equation}\label{eq:qmult}
(a_1 + a_2 \xi) (b_1 + b_2 \xi) = 
\begin{cases}
(a_1b_1 - \tau a_2b_2)+ (a_1 b_2 + a_2b_1) \xi &\text{if } f(x) = x^2 + \tau, \\
(a_1b_1 - \tau a_2b_2)+ (a_1 b_2 + a_2b_1 - a_2b_2) \xi &\text{if } f(x) = x^2 + x + \tau.
\end{cases}
\end{equation}
The following result is well-known for $\mathbb{C} = \mathbb{R}[i]$ but we are unable to find a reference for an arbitrary quadratic extension $\Bbbk[\xi]$.
\begin{proposition}[Complexity of multiplication in quadratic extensions]\label{prop:mult}
Let $\Bbbk, f,\tau,\xi$ be as above. Then there exists an algorithm for multiplication in $\mathbb{F} = \Bbbk[\xi]$ that costs three multiplications in $\Bbbk$. Moreover, such an algorithm is optimal in the sense of bilinear complexity, i.e., it requires a minimal number of multiplications in $\Bbbk$.
\end{proposition}
\begin{proof}
\underline{Case I:} $f(x) = x^2 + \tau$. The product in \eqref{eq:qmult} can be computed with three $\Bbbk$-multiplications $m_1 = (a_1 - a_2)(b_1 + \tau b_2 )$, $m_2 = a_1 b_2$, $m_3 = a_2b_1$, since 
\begin{equation}\label{eq:prop:mult:case1}
a_1 b_1 - \tau a_2 b_2 = m_1 - \tau m_2 + m_3, \quad a_1b_2 + a_2b_1 = m_2 + m_3.
\end{equation}
\underline{Case II:}   $f(x) = x^2 + x + \tau$. The product in \eqref{eq:qmult} can be computed with three $\Bbbk$-multiplications $m_1 = a_1b_1$, $m_2 = a_2b_2$, $m_3 = (a_1- a_2)(b_1 -b_2)$, since
\begin{equation}\label{eq:prop:mult:case2}
a_1 b_1 - \tau a_2 b_2 = m_1 - \tau m_2,\quad a_1b_2 + a_2b_1 - a_2b_2 = m_1 - m_3. 
\end{equation}
To show optimality in both cases suppose there is an algorithm for computing \eqref{eq:qmult} with two $\Bbbk$-multiplications $m_1'$ and $m_2'$. Then
\[
a_1b_1 - \tau a_2 b_2, \; a_1b_2 + a_2b_1 - \delta a_2 b_2 \in \spn\{m'_1, m'_2\},
\]
where $\delta = 0$ in Case~I and $\delta = 1$ in Case~II. Clearly $a_1b_1 - \tau a_2 b_2$ and $a_1b_2 + a_2b_1 - \delta a_2 b_2$ are not collinear; thus 
\[
m'_1, m'_2  \in \spn\{a_1b_1 - \tau a_2 b_2, a_1b_2 + a_2b_1 - \delta a_2 b_2\}
\]
and so there exist $p,q,r,s\in \Bbbk$,  $p s - q r \ne 0$, such that 
\begin{align*}
m'_1 &= p (a_1b_1 - \tau a_2 b_2) + q(a_1b_2 + a_2b_1 - \delta a_2 b_2) = p a_1 b_1 + q a_1b_2 + q a_2b_1 +(-\tau p -\delta q) a_2b_2, \\ 
m'_2 &= r (a_1b_1 - \tau a_2 b_2) + s (a_1b_2 + a_2b_1 - \delta a_2 b_2) = r a_1 b_1 + s a_1b_2 + s a_2b_1 +(-\tau r -\delta s) a_2b_2.
\end{align*}
As $p s - q r \ne 0$, at least one of $p,q,r,s$ is nonzero. Since $m_1'$ is a $\Bbbk$-multiplication, we must have $m_1' = (\lambda_1 a_1 + \lambda_2 a_2)(\mu_1 b_1 + \mu_2 b_2)$ for some $\lambda_1 a_1 + \lambda_2 a_2$, $\mu_1 b_1 + \mu_2 b_2 \in \Bbbk$. Therefore
\[
p(-\tau p - \delta q) = q^2, \qquad
r(-\tau r - \delta s) = s^2.
\]
For Case~I, the left equation reduces to $\tau p^2 + q^2 = 0$ and thus $p = q = 0$ as $-\tau$ is not a complete square in $\Bbbk$; likewise, the right equation gives $r = s = 0$, a contradiction as $p,q,r,s$ cannot be all zero. 
For Case~II,  the left equation reduces to $\tau p^2 + p q +  q^2 = 0$. We must have $p \ne 0$ or else $q = 0$ will contradict $p s -q r \ne 0$; but if so, substituting $q' = q/p$ gives ${q'}^2 + q' + \tau=0$, contradicting the assumption that $x^2 + x + \tau = 0$ has no solution in $\Bbbk$. 
\end{proof}
For the special case when $\Bbbk = \mathbb{R}$ and $f(x) = x^2 + 1$, we have $\xi = i$ and $\mathbb{F} = \Bbbk[\xi] = \mathbb{C}$ and the algorithm in \eqref{eq:prop:mult:case1} is the celebrated Gauss  multiplication of complex numbers, $(a_1 + i a_2)(b_1+ i b_2) =  (a_1 b_1 - a_2 b_2) + i[(a_1+a_2)(b_1+b_2) -a_1 b_1 -a_2 b_2]$,  whose optimality is proved in \cite{Munro71,winograd70}.  Proposition~\ref{prop:mult} may be viewed as a generalization of Gauss multiplication to arbitrary quadratic extensions.

In the language of tensors \cite[Example~3.8]{acta}, multiplication in $\Bbbk[\xi]$ is a bilinear map over $\Bbbk$,
\[
m: \Bbbk[\xi] \times \Bbbk[\xi] \to \Bbbk[\xi], \quad (a_1 + a_2 \xi, b_1 + b_2 \xi) \mapsto (a_1 + a_2 \xi) (b_1 + b_2 \xi),
\]
and therefore corresponds to a tensor in $\mu \in \Bbbk[\xi] \otimes \Bbbk[\xi] \otimes \Bbbk[\xi]$.  An equivalent way to state Proposition~\ref{prop:mult} is that the tensor rank of $\mu$ is exactly three.

\subsection{Gauss matrix multiplication over quadratic field extensions}\label{sec:gauss mult}

We extend the multiplication algorithm in the previous section to matrices. Notations will be as in the last section.  Let $\mathbb{F}^{n \times n}$ be the $\mathbb{F}$-algebra of $n\times n$ matrices over $\mathbb{F}$.  Since $\mathbb{F} = \Bbbk[\xi]$, we have $\mathbb{F}^{n \times n} = \Bbbk^{n \times n} \otimes_{\Bbbk} \mathbb{F}$ \cite[p.~627]{acta}. Thus an element in $X \in \mathbb{F}^{n \times n}$ can be written as $X = A + \xi B$ where $A,B\in \Bbbk^{n \times n}$.

By following the argument in the proof of Proposition~\ref{prop:mult}, we obtain its analogue for matrix multiplication in $\mathbb{F}^{n \times n}$ via matrix multiplications in $\Bbbk^{n \times n}$.
\begin{proposition}[Gauss matrix multiplication]\label{prop:matrix mult}
Let $\Bbbk,\mathbb{F}, n,f,\tau, \xi$ be as before. Let  $X = A + \xi B$, $Y = C + \xi D \in \mathbb{F}^{n \times n}$ with $A,B,C,D\in \Bbbk^{n \times n}$. If $f(x) = x^2 + \tau$, then $XY$ can be computed via
\begin{equation}\label{eq:matrix mult-1}
\begin{aligned}
M_1 &= (A - B)(C + \tau D), &M_2 &= AD,  &M_3 &= BC; \\
N_1 &= M_1 - \tau M_2 + M_3,  &N_2 &= M_2 + M_3;  &XY &=N_1 + \xi N_2.
\end{aligned}
\end{equation}
If $f(x) = x^2 + x + \tau$, then $XY$ can be computed via
\begin{equation}\label{eq:matrix mult-2}
\begin{aligned}
M_1 &= AC,  &M_2 &= BD, &M_3 &= (A-B)(C-D);\\
N_1 &= M_1 - \tau M_2, &N_2 &= M_1 - M_3; &XY &= N_1 + \xi N_2. 
\end{aligned}
\end{equation}
The algorithms for forming $XY$ in \eqref{eq:matrix mult-1} and \eqref{eq:matrix mult-2} use a minimal number of matrix multiplications in $\Bbbk^{n \times n}$.
\end{proposition}
\begin{proof}
It is straightforward to check that \eqref{eq:matrix mult-1} and \eqref{eq:matrix mult-2} give $XY$. To see minimality, we repeat the proof of Proposition~\ref{prop:mult} noting that the argument depends only on $\mathbb{F}$ as a two-dimensional free $\Bbbk$-module, and that $\mathbb{F}^{n \times n}$ is also a two-dimensional free $\Bbbk^{n \times n}$-module.
\end{proof}

\subsection{Frobenius matrix inversion over quadratic field extensions}\label{sec:Frob inv}

Let $A + \xi B \in \GL_n(\mathbb{F})$ with $A,B \in \Bbbk^{n \times n}$. Then $(A + \xi B)^{-1} = C + \xi D$ if and only if 
\begin{equation}\label{eq:inverse equation}
(A + \xi B) (C + \xi D)  = I,
\end{equation} 
from which we may solve for  $C,D \in \Bbbk^{n \times n}$. As we saw in \eqref{eq:qmult}, multiplication in $\mathbb{F}$ and thus that in $\mathbb{F}^{n \times n}$ depends on the form of $f$. So we have to consider two cases corresponding to the two normal forms of $f$. 

\begin{lemma}\label{lemma:invertibility}
Let $\Bbbk,\mathbb{F}, n,f,\tau, \xi$ be as before. Let $A + \xi B\in \GL_n(\mathbb{F})$ with $A,B\in \Bbbk^{n \times n}$.
\begin{enumerate}[\normalfont(i)]
\item  If $f(x) = x^2 + \tau$, then $A + \tau B A^{-1} B \in \GL_n(\Bbbk)$ whenever $A\in \GL_n(\Bbbk)$. \label{lemma:invertibility:eq1}
\item If $f(x) = x^2 + x + \tau$, then $\tau B + A B^{-1} A - A \in \GL_n(\Bbbk)$ whenever $B \in \GL_n(\Bbbk)$. \label{lemma:invertibility:eq4} \label{lemma:invertibility:eq3}
\end{enumerate}
\end{lemma}
\begin{proof}
Consider the case $f(x) = x^2 + \tau$. By \eqref{eq:inverse equation},
$AC - \tau BD = I$ and $AD + BC = 0$. So $(A + \tau B A^{-1} B )C = I$.
Hence $A + \tau B A^{-1} B$ is invertible. A similar argument applies to the case $f(x) = x^2 + x+ \tau$ to yield \ref{lemma:invertibility:eq4}.
\end{proof}

Let the matrix addition, multiplication, and inversion maps over any field $\mathbb{F}$ be denoted respectively by
\begin{alignat*}{2}
\add_{n,\mathbb{F} }: \mathbb{F}^{n \times n} \times \mathbb{F}^{n \times n} &\to \mathbb{F}^{n \times n}, \qquad &\add_{n,\mathbb{F} }(X,Y) &= X+Y;\\
\mul_{n,\mathbb{F} }: \mathbb{F}^{n \times n} \times \mathbb{F}^{n \times n} &\to \mathbb{F}^{n \times n}, \qquad &\mul_{n,\mathbb{F} }(X,Y) &= XY;\\
\inv_{n,\mathbb{F}}: \GL_n(\mathbb{F}) &\to \GL_n(\mathbb{F}),  &\inv_{n,\mathbb{F}}(X) &= X^{-1}.
\end{alignat*}
We will now express $\inv_{n,\mathbb{F}}$ in terms of $\inv_{n,\Bbbk}$, $\mul_{n,\Bbbk}$, and $\add_{n,\Bbbk}$.

\begin{lemma}[Frobenius inversion over quadratic fields]\label{lem:inverse}
Let $\Bbbk,\mathbb{F}, n,f,\tau, \xi$ be as before. Let $X = A + \xi B\in \GL_n(\mathbb{F})$ with $A,B\in \Bbbk^{n \times n}$. If $f(x) = x^2 + \tau$ and $A \in \GL_n(\Bbbk)$, then
\begin{equation}\label{eq:lem:inverse}
    X^{-1} = (A + \tau BA^{-1}B)^{-1} - \xi A^{-1}B(A+ \tau BA^{-1} B)^{-1}.
\end{equation}
If $f(x) = x^2 + x + \tau$ and $B \in \GL_n(\Bbbk)$, then
\begin{equation}\label{eq:lem:inverse2}
X^{-1} = (B^{-1}A - I)(AB^{-1}A - A + \tau B)^{-1} - \xi (AB^{-1}A - A + \tau B)^{-1}
\end{equation}
\end{lemma}
\begin{proof}
\underline{Case I:} $f(x) = x^2 + \tau$. From \eqref{eq:inverse equation}, we get
\[
AC - \tau BD = I,\qquad AD + BC = 0.
\]
\underline{Case II:} $f(x) = x^2 + x + \tau$. From \eqref{eq:inverse equation}, we get
\[
AC  - \tau B D = I,\qquad A D + BC - BD = 0.
\]
In each case, solving the equations for $C$ and $D$ gives us the required expressions \eqref{eq:lem:inverse} and \eqref{eq:lem:inverse2}.
\end{proof}
We could derive alternative inversion formulas with other conditions on $A$ and $B$. For example, in the case $f(x) = x^2 + \tau$, instead of \eqref{eq:lem:inverse}, we could have
\[
X^{-1} = B^{-1}A(AB^{-1}A + \tau B)^{-1} - \xi (AB^{-1}A + \tau B)^{-1},
\]
conditional on $B \in \GL_n(\Bbbk)$; in the case $f(x) = x^2 + x + \tau$,  instead of \eqref{eq:lem:inverse2}, we could have
\[
X^{-1} = (A + \tau B (A-B)^{-1} B)^{-1} - \xi (A-B)^{-1} B (A + \tau B (A-B)^{-1} B)^{-1},
\]
conditional on $A-B \in \GL_n(\Bbbk)$. There is no single inversion formula that will work universally for all $A + \xi B \in \GL_n(\mathbb{F})$. Nevertheless, in each case, the inversion formula \eqref{eq:lem:inverse} or \eqref{eq:lem:inverse2} works almost everywhere except for matrices $A + \xi B$ with $\det(A) = 0$ or $\det(B) = 0$ respectively. In Section~\ref{sec:rand}, we will see how to alleviate this minor restriction algorithmically for complex matrices.

We claim that \eqref{eq:lem:inverse} and \eqref{eq:lem:inverse2} allow $\inv_{n,\mathbb{F}}$ to be evaluated by invoking $\inv_{n,\Bbbk}$ twice, $\mul_{n,\Bbbk}$ thrice, and $\add_{n,\Bbbk}$ once. To see this more clearly, we express them in pseudocode as Algorithms~\ref{alg:inverse} and \ref{alg:inverse2} respectively.
\begin{algorithm}[!htb]
\caption{Frobenius Inversion with $\xi$ a root of $ x^2 + \tau$}
\label{alg:inverse}
\begin{algorithmic}[1]
\Require $X = A + \xi B$ with $A \in \GL_n(\Bbbk)$
\State matrix invert $X_1 = A^{-1}$;
\State matrix multiply $X_2 = X_1B$;
\State matrix multiply $X_3 = B X_2$;
\State matrix add $X_4 = A + \tau X_3$;
\State matrix invert $X_5 =  X_4^{-1}$;
\State matrix multiply $X_6 = X_2 X_5$;
\Ensure inverse $X^{-1} = X_5 - \xi X_6$
\end{algorithmic}
\end{algorithm}

A few words are in order here. A numerical linear algebraist may balk at inverting $A$ and then multiplying it to $B$ to form $A^{-1}B$ instead of solving a linear system with multiple right-hand sides. However, Algorithms~\ref{alg:inverse} and \ref{alg:inverse2} should be viewed in the context of \emph{symbolic computing} over arbitrary fields. To establish complexity results like Theorem~\ref{thm:optimal inverse} and Theorem~\ref{thm:GFiter}, we would have to state the algorithms purely in terms of algebraic operations in $\Bbbk^{n \times n}$, i.e., $\inv_{n,\Bbbk}$, $\mul_{n,\Bbbk}$, and $\add_{n,\Bbbk}$.
The \emph{numerical computing} aspects specific to $\Bbbk = \mathbb{R}$ and $\mathbb{F}= \mathbb{C}$ will be deferred to Sections~\ref{sec:explicit}--\ref{sec:expr}, where, among other things, we would present several numerical computing variants of Algorithm~\ref{alg:inverse} (see Algorithms~\ref{alg:Frob:linear}, \ref{alg:inverse3}, \ref{alg:rand inverse}, \ref{alg:inverse4}). We also remind the reader that a term like $X_5 - \xi X_6$ in the output of these algorithms does not entail  matrix addition; here $\xi$ plays a purely symbolic role like the imaginary unit $i$, and $X_5$ and $-X_6$ are akin to the `real part' and `imaginary part.'

\begin{algorithm}[!htb]
\caption{Frobenius Inversion with $\xi$ a root of $ x^2 + x + \tau$}
\label{alg:inverse2}
\begin{algorithmic}[1]
\Require $X = A + \xi B$  with $B \in \GL_n(\Bbbk)$
\State matrix invert $X_1 = B^{-1}$;
\State matrix multiply $X_2 = X_1A -I$;
\State matrix multiply $X_3 = AX_2$;
\State matrix add $X_4 =  X_3 + \tau B$;
\State matrix invert $X_5 = X_4^{-1}$;
\State matrix multiply $X_6 = X_3 X_5$;
\Ensure inverse $X^{-1} = X_6 - \xi X_5$
\end{algorithmic}
\end{algorithm}

Note that the addition of a fixed constant (i.e., independent of inputs $A$ and $B$) matrix $-I$ in Step~2 of Algorithm~\ref{alg:inverse2} does not count towards the computational complexity of the algorithm \cite{algebraiccomplexity}.

As we mentioned earlier, $\mathbb{F}^{n \times n}$ is a $\Bbbk^{n \times n}$-bimodule. We prove next that Algorithms~\ref{alg:inverse}  and \ref{alg:inverse2} have optimal computational complexity in terms of matrix operations in $\Bbbk^{n \times n}$.
\begin{theorem}[Optimality of Frobenius Inversion]\label{thm:optimal inverse}
Algorithm~\ref{alg:inverse} and \ref{alg:inverse2} for $\inv_{n,\mathbb{F}}$ require the fewest number of matrix operations in $\Bbbk^{n \times n}$: two  $\inv_{n,\Bbbk}$, three $\mul_{n,\Bbbk}$, and one $\add_{n,\Bbbk}$, i.e., there is no algorithm for  matrix inversion in $\mathbb{F}^{n \times n}$ that takes four or fewer matrix operations in $\Bbbk^{n \times n}$.
\end{theorem}
\begin{proof}
If $n = 1$, then this reduces to Proposition~\ref{prop:mult}. So we will assume that $n \ge 2$. We will restrict ourselves to Algorithm~\ref{alg:inverse} as the argument for Algorithm~\ref{alg:inverse2} is nearly identical.

Clearly, we need at least one $\add_{n,\Bbbk}$ to compute $\inv_{n,\mathbb{F}}$ so  Algorithm~\ref{alg:inverse} is already optimal in this regard. We just need to restrict ourselves to the numbers of $\inv_{n,\Bbbk}$ and $\mul_{n,\Bbbk}$, which are invoked twice and thrice respectively in Algorithm~\ref{alg:inverse}. We will show that these numbers are minimal. In the following, we pick any $A,B \in \GL_n(\Bbbk)$ that do not commute.

First we claim that it is impossible to compute  $(A + \xi B)^{-1}$ with fewer than two $\inv_{n,\Bbbk}$ even with no limit on the number of $\add_{n,\Bbbk}$ and $\mul_{n,\Bbbk}$.  By \eqref{eq:lem:inverse},  $(A + \xi B)^{-1}$  comprises two $\Bbbk^{n \times n}$ matrices $(A + \tau BA^{-1}B)^{-1}$ and $A^{-1}B (A + \tau BA^{-1}B)^{-1}$, which we will call its `real part' and `imaginary part' respectively, slightly abusing terminologies.  We claim that computing the `real part' $(A + \tau BA^{-1}B)^{-1}$ alone already takes at least  two $\inv_{n,\Bbbk}$. 
If $(A+\tau BA^{-1}B)^{-1}$ can be computed with just one $\inv_{n,\Bbbk}$,  then $A(A+\tau BA^{-1}B)^{-1}$ can also be computed with just one $\inv_{n,\Bbbk}$ as the extra factor $A$ involves no inversion. However, if it takes only one $\inv_{n,\Bbbk}$, then we must have an expression
\[
A(A+\tau BA^{-1}B)^{-1} = f(A, B, g(A,B)^{-1})
\]
for some noncommutative polynomials $f \in \Bbbk \langle x,y,z \rangle$ and $g \in \Bbbk\langle x,y \rangle$. Now observe that 
\[
A (A+ \tau BA^{-1}B)^{-1}  = (I + \tau (B A^{-1})^2)^{-1}.
\]
To see that the last two expressions are contradictory, we write $X \coloneqq BA^{-1}$ and expand them in formal power series, thereby removing negative powers for an easier comparison:
\[
\sum_{k=0}^\infty (-\tau)^k X^{2k} =  (I + \tau X^2)^{-1} = f(A, XA, g(A,XA)^{-1}) = f\Bigl(A, XA, \sum_{k=0}^\infty \bigl( I - g(A,XA)\bigr)^k \Bigr).
\]
Note that the leftmost expression is purely in powers of $X$, but the rightmost expression must necessarily involve $A$ --- indeed any term involving a power of $X$ must involve $A$ to the same or higher power. The remaining possibility that $X$ is a power of $A$ is excluded since $A$ and $B$ do not commute. So we arrive at a contradiction. Hence $(A + \tau BA^{-1}B)^{-1}$ and therefore  $(A + \xi B)^{-1}$ requires at least two $\inv_{n,\Bbbk}$ to compute.

Next we claim that it is impossible to compute  $(A + \xi B)^{-1}$ with fewer than three $\mul_{n,\Bbbk}$ even with no limit on the number of  $\add_{n,\Bbbk}$ and $\inv_{n,\Bbbk}$. Let the `real part' and `imaginary part' be denoted
\[
Y \coloneqq (A + \tau BA^{-1}B)^{-1}, \qquad Z \coloneqq A^{-1}B (A + \tau BA^{-1}B)^{-1} = (B+ \tau AB^{-1}A)^{-1}.
\]
Observe that we may express $BA^{-1}B $ in terms of $Y$ and $AB^{-1}A $ in terms of $Z$ using only  $\add_{n,\Bbbk}$ and $\inv_{n,\Bbbk}$:
\[
BA^{-1}B =\tau^{-1} ( Y^{-1} - A ), \qquad  AB^{-1}A = \tau^{-1}(Z^{-1} - B).
\]
So computing both $BA^{-1}B$  and $A B^{-1}A$ take the same number of  $\mul_{n,\Bbbk}$ as computing both $Y$ and $Z$.  However, as $A$ and $B$ do not commute, it is impossible to compute both $BA^{-1}B$ and $AB^{-1}A$ with just two $\mul_{n,\Bbbk}$.  Consequently $(A + \xi B)^{-1} = Y + \xi Z$ requires at least three $\mul_{n,\Bbbk}$ to compute.
\end{proof}
A more formal way to cast our proof above would involve the notion of a \emph{straight-line program} \cite[Definition~4.2]{algebraiccomplexity}, but we prefer to avoid pedantry given that the ideas involved are the same.

\subsection{Frobenius inversion over iterated quadratic extensions}\label{sec:tower}

Repeated applications of Algorithms~\ref{alg:inverse} and \ref{alg:inverse2} allow us to extend Frobenius inversion to an \emph{iterated quadratic extension}: 
\begin{equation}\label{eq:tower}
\Bbbk \eqqcolon \mathbb{F}_0 \subsetneq \mathbb{F}_1 \subsetneq \cdots \subsetneq \mathbb{F}_m \coloneqq \mathbb{F},
\end{equation}
where $[\mathbb{F}_k : \mathbb{F}_{k-1}] = 2$, $k=1,\dots, m$. By our discussion at the beginning of Section~\ref{sec:symb},  $\mathbb{F}_k = \mathbb{F}_{k-1}[\xi_k]$ for some  $\xi_k\in \mathbb{F}_k$. Let $f_k \in \Bbbk[x]$  be the minimal polynomial \cite{Roman} of $\xi_k$. Then $f_k$ is a monic irreducible quadratic polynomial that we may assume is in normal form, i.e.,
\[
f_k(x) = x^2 + \tau_k\quad \text{or} \quad f_k(x) = x^2 + x + \tau_k,\quad k=1,\dots, m.
\]
Since $[\mathbb{F} : \Bbbk]  = \prod_{k=1}^m = [\mathbb{F}_k : \mathbb{F}_{k-1}] = 2^m$, any element in  $\mathbb{F}$ may be written as
\begin{equation}\label{eq:iterquad}
\sum_{\alpha \in \{0,1\}^{m}} c_\alpha \xi^{\alpha}
\end{equation}
in \emph{multi-index} notation with $\alpha= (\alpha_1,\dots, \alpha_{m})  \in \{0,1\}^m$, $\xi^\alpha \coloneqq \xi_1^{\alpha_1} \cdots \xi_{m}^{\alpha_{m}}$, and $c_\alpha \in \Bbbk$. Moreover, we may regard $\mathbb{F}$ as a quotient ring of a multivariate polynomial ring  or as a tensor product of $m$ quotient rings of univariate polynomial ring:
\begin{equation}\label{eq:K_m quotient}
\mathbb{F} \simeq \Bbbk[x_1,\dots, x_m]\!\bigm/\!\! \langle f_1,\dots, f_m \rangle = \bigotimes_{k=1}^m \bigl( \Bbbk[x] \!\bigm/\!\! \langle f_k \rangle \bigr).
\end{equation}
There are many important fields that are special cases of iterated quadratic extensions
\begin{example}[Constructible numbers]
One of the most famous instance is the special case $\Bbbk = \mathbb{Q}$ with the iterated quadratic extension $\mathbb{F} \subseteq \mathbb{R}$.  In which case the positive numbers in $\mathbb{F}$ are called \emph{constructible numbers} and they are precisely the lengths that can be constructed with a compass and a straightedge in a finite number of steps. The impossibillity of trisecting an angle, doubling a cube, squaring a circle, constructing $n$-sided regular polygons for $n =  7, 9, 11, 13, 14, 18, \dots$, etc, were all established using the notion of constructible numbers.
\end{example}

\begin{example}[Multiquadratic fields]
Another interesting example is $\mathbb{F} = \mathbb{Q}[\sqrt{q_1},\dots, \sqrt{q_m} ]$. It is shown in \cite{Besicovitch40} that 
\[
\mathbb{Q} \subsetneq \mathbb{Q}[\sqrt{q_1}] \subsetneq \mathbb{Q}[\sqrt{q_1},\sqrt{q_2}] \subsetneq \cdots \subsetneq \mathbb{Q}[\sqrt{q_1},\dots, \sqrt{q_m} ]
\]
is an iterated quadratic extension if the product of any nonempty subset of $\{\sqrt{q_1},\dots, \sqrt{q_m} \}$ is not in $\mathbb{Q}$. In this case, we have $\mathbb{F}_k = \mathbb{Q}[\sqrt{q}_1,\dots, \sqrt{q}_k]$ and $f_k(x) = x^2 - q_k$, $k = 1,\dots, m$.
\end{example}

\begin{example}[Tower of root extensions of non-square]
Yet another commonly occurring example \cite[Section~14.7]{DF04} of iterated quadratic extension is a `tower' of root extensions:
\[
\mathbb{Q} \subsetneq \mathbb{Q}[q^{1/2}] \subsetneq \mathbb{Q}[q^{1/4}] \subsetneq \cdots \subsetneq \mathbb{Q}[q^{1/2^m}]
\]
where $q\in \mathbb{Q}$ is not a complete square.
\end{example}

Since $\Bbbk^{n \times n}$ and $\mathbb{F}$ are both free $\Bbbk$-modules, we have  $\mathbb{F}^{n \times n} = \Bbbk^{n \times n} \otimes_{\Bbbk} \mathbb{F}$ as tensor product of $\Bbbk$-modules. Hence the expression \eqref{eq:iterquad} may be extended to matrices, i.e., any $X \in \mathbb{F}^{n \times n}$ may be written as 
\begin{equation}\label{eq:iterquadmat}
X = \sum_{\alpha \in \{0,1\}^m} C_\alpha \xi^\alpha
\end{equation}
with $C_\alpha \in \Bbbk^{n \times n}$, $\alpha  \in \{0,1\}^m$. Note that the $c_\alpha$ in \eqref{eq:iterquad} are scalars and the $C_\alpha$ in \eqref{eq:iterquadmat} are matrices. On the other hand, in an iterated quadratic extension \eqref{eq:tower}, each $\mathbb{F}_k$ is an $\mathbb{F}_{k-1}$-module, $k=1,\dots,m$, and thus we also have the tensor product relation
\[
\Bbbk^{n \times n} \otimes_{\Bbbk} \mathbb{F} = \Bbbk^{n \times n} 
\otimes_{\Bbbk} \mathbb{F}_1 \otimes_{\mathbb{F}_1} \mathbb{F}_2 \otimes_{\mathbb{F}_2}\cdots \otimes_{\mathbb{F}_{m-1}} \mathbb{F},
\]
recalling that $\mathbb{F}_0 \coloneqq \Bbbk$ and $\mathbb{F}_m \coloneqq \mathbb{F}$. Hence any $X\in \mathbb{F}^{n \times n}$ may also be expressed recursively as
\begin{equation}\label{eq:iterquadmat2}
\begin{aligned}
X &= A_0 + \xi_{m} A_1,\\
A_\beta &= A_{0,\beta} + \xi_{m-k} A_{1,\beta}, &\beta &\in \{0,1\}^{k}, &k&=1,\dots, m-1, 
\end{aligned}
\end{equation}
with $A_\beta \in \mathbb{F}^{n \times n}_{m - |\beta|}$. The relation between the two expressions \eqref{eq:iterquadmat} and \eqref{eq:iterquadmat2} is given as follows.
\begin{lemma}\label{lem:inductive matrix over Km-2}
Let $X \in \mathbb{F}^{n \times n}$ be expressed as in \eqref{eq:iterquadmat} with $C_\alpha \in \Bbbk^{n \times n}$, $\alpha \in \{0,1\}^{m}$, and as in \eqref{eq:iterquadmat2} with $A_{\beta}  \in \mathbb{F}^{n \times n}_{m - |\beta|}$, $\beta\in \{0,1\}^k$. Then for any $k \in \{ 1,\dots,m\}$,
\[
X = \sum_{\beta \in \{0,1\}^k} A_\beta \xi_{m-k+1}^{\beta_1} \cdots \xi_{m}^{\beta_k},
\] 
and for any $\beta\in \{0,1\}^k$,
\[
A_{\beta} = \sum_{\gamma\in \{0,1\}^{m-k} } C_{\gamma,\beta} \xi_1^{\gamma_1}\cdots \xi_{m-k}^{\gamma_{m-k}}.
\]
In particular, $C_\alpha = A_\alpha$.
\end{lemma}
\begin{proof}
We proceed by induction on $k$. Clearly the formula holds for $k = 1$ by \eqref{eq:iterquadmat2}. Assume that the first expression holds for $k = s$, i.e.,
\[
X = \sum_{\beta \in \{0,1\}^s } A_\beta \xi_{m-s+1}^{\beta_1} \cdots \xi_{m}^{\beta_s}.
\]
To show that it also holds for $k = s+1$, note that $A_\beta = A_{0,\beta} + \xi_{m - s}A_{1,\beta}$, so
\[
X = \sum_{\beta \in \{0,1\}^s } (A_{0,\beta} + \xi_{m - s}A_{1,\beta}) \xi_{m-s+1}^{\beta_1} \cdots \xi_{m}^{\beta_s} = \sum_{\gamma \in \{0,1\}^{s+1}} A_\gamma \xi_{m-s}^{\gamma_1} \cdots \xi_{m}^{\gamma_{s+1}}
\]
completing the induction. Comparing coefficients in \eqref{eq:iterquadmat} and \eqref{eq:iterquadmat2} yields the second expression.
\end{proof}
The representation in Lemma~\ref{lem:inductive matrix over Km-2}, when combined with Gauss multiplication, gives us a method for fast matrix multiplication  in $\mathbb{F}^{n \times n}$, and, when combined with Frobenius inversion, gives us a method for fast matrix inversion in $\mathbb{F}^{n \times n}$. 
\begin{theorem}[Gauss multiplication and Frobenius inversion over iterated quadratic extension]\label{thm:GFiter}
Let $\mathbb{F}$ be an iterated quadratic extension of $\Bbbk$ of degree $[\mathbb{F} : \Bbbk] = 2^m$. Then
\begin{enumerate}[\normalfont(i)]
\item one may multiply two matrices in $\mathbb{F}^{n \times n}$ with $3^m$ multiplications in $\Bbbk^{n \times n}$;
\item one may invert a generic matrix in $\mathbb{F}^{n \times n}$ with $3(3^m - 2^m)$ multiplications and $2^m$ inversions in $\Bbbk^{n \times n}$.
\end{enumerate}
If we write $N = 2^m$, this multiplication algorithm reduces the complexity of evaluating $\mul_{n,\mathbb{F}}$ from $O(N^2)$ to $O(N^{\log_2 3})$ $\mul_{n,\Bbbk}$.
\end{theorem}
\begin{proof}
Let $X,Y\in \mathbb{F}^{n \times n}$. By \eqref{eq:iterquadmat2}, we may write 
\[
X = A_0 + \xi_m A_1,\quad Y = B_0 + \xi_m B_1,
\]
and thus compute $XY$ in terms of $A_0,A_1,B_0,B_1 \in \mathbb{F}^{n \times n}_{m-1}$ using three $\mul_{n,\mathbb{F}_{m-1}}$ by Proposition~\ref{prop:matrix mult}. Each $\mul_{n,\mathbb{F}_{m-1}}$ in turn costs three $\mul_{n,\mathbb{F}_{m-2}}$ by Proposition~\ref{prop:matrix mult}. Repeating the argument until we arrive at $\mul_{n,\mathbb{F}_0} = \mul_{n,\Bbbk}$, we see that the total number of $\mul_{n,\Bbbk}$ is $3^m$.

Now if $X$ above is generic, depending on whether $f_m(x) = x^2 + \tau_m$ or $x^2 + x + \tau_m$, Algorithm~\ref{alg:inverse} or Algorithm~\ref{alg:inverse2} takes three $\mul_{n,\mathbb{F}_{m-1}}$ and two $\inv_{n,\mathbb{F}_{m-1}}$. As in the multiplication case, the argument applies recursively to $m,m-1,\dots,2,1$. Writing $\num(\op)$ for the number of operation $\op$, we have 
\[
\num(\inv_{n,\mathbb{F}_k}) = 3 \num(\mul_{n,\mathbb{F}_{k-1}}) + 2 \num(\inv_{n,\mathbb{F}_{k-1}}),\qquad k = 1,\dots,m.
\]
By Proposition~\ref{prop:matrix mult}, we have $\num(\mul_{n,\mathbb{F}_k}) = 3 \num(\mul_{n,\mathbb{F}_{k-1}})$. Hence we obtain
\[
\num(\inv_{n,\mathbb{F}}) = 3(3^m - 2^m) \num(\mul_{n,\Bbbk}) + 2^m \num(\inv_{n,\Bbbk})
\]
as required.
\end{proof}
Slightly abusing terminologies, we will call the multiplication and inversion algorithms in the proof of Theorem~\ref{thm:GFiter} Gauss multiplication and Frobenius inversion for iterated quadratic extension respectively. Both rely on the general technique of \emph{divide-and-conquer}. Moreover, Gauss multiplication for iterated quadratic extension is in spirit the same as the Karatsuba algorithm \cite{Karatsuba95} for fast integer multiplication and multidimensional fast Fourier transform \cite{SS95}. Indeed, all three algorithms may be viewed as fast algorithms for modular polynomial multiplication in a ring
\[
\Bbbk[x_1,\dots, x_m]/\langle x_1^d + \tau_1,\dots, x_m^d + \tau_m \rangle \simeq \Bbbk[x_1]/\langle x_1^d + \tau_1 \rangle \otimes  \cdots \otimes \Bbbk[x_m]/\langle x_m^d + \tau_m \rangle
\]
for different choices of $m,d, \tau_1,\dots, \tau_m$. To be more specific, we have 
\begin{enumerate}[\normalfont(a)]
\item Karatsuba algorithm: $m = 1$, $d =0$, $\tau_1  = -1$.
\item Multidimensional fast Fourier transform: $m \in \mathbb{N}$, $d \in \mathbb{N}$, $\tau_1 = \cdots = \tau_m = -1$.
\item Gauss multiplication for iterated quadratic: $m \in \mathbb{N}$, $d = 2$, $\tau_1,\dots, \tau_m$ as in \eqref{eq:tower}--\eqref{eq:K_m quotient}.
\end{enumerate}

\subsection{Moore--Penrose and Sherman--Morrison}

One might ask if Frobenius inversion extends  to pseudoinverse. In particular, does  \eqref{eq:lem:inverse} hold if matrix inverse is replaced by Moore--Penrose inverse? The answer is no, which can be seen by taking $\Bbbk = \mathbb{R}$, $\mathbb{F}= \mathbb{C}$, in which case \eqref{eq:lem:inverse} is just \eqref{eq:Finv}. Let
\[
X = \begin{bmatrix}
    0 & 0 & 0 \\
    0 & 1 & 0 \\
    0 & 0 & i
\end{bmatrix}, \qquad
X^\dagger = 
\begin{bmatrix}
    0 & 0 & 0 \\
    0 & 1 & 0 \\
    0 & 0 & -i
\end{bmatrix},  \qquad
Y = \begin{bmatrix}
    0 & 0 & 0 \\
    0 & 1 & 0 \\
    0 & 0 & 0
\end{bmatrix},
\]
where $Y$ denotes the right-hand side of \eqref{eq:Finv} with Moore--Penrose inverse in place of matrix inverse. Clearly $X^\dagger \ne Y$.

One may be led to think that Frobenius inversion \eqref{eq:lem:inverse} is a consequence of Sherman--Morrison--Woodbury-type identities such as
\begin{align*}
\left(A + B\right)^{-1} &= A^{-1} - A^{-1} (B^{-1} + A^{-1})^{-1} A^{-1}= A^{-1} - A^{-1}\left(AB^{-1} + I\right)^{-1}\\
&= A^{-1} - \left(A + AB^{-1}A\right)^{-1}= A^{-1} - A^{-1}B\left(A + B\right)^{-1}
\end{align*}
but it is not. The point to note is that such identities invariably involve at least one matrix inversion in $\mathbb{F}^{n \times n}$ whereas \eqref{eq:lem:inverse} is purely in terms of matrix inversions in $\Bbbk^{n \times n}$.

\section{Solving linear systems with Frobenius inversion}\label{sec:lin}

We remind the reader that the previous section is the only one about Frobenius inversion over arbitrary fields. In this and all subsequent sections we return to the familiar setting of real and complex fields. In this section we discuss the solution of a system of complex linear equations
\begin{equation}\label{eq:linear system}
(A+iB)(x+iy) = c+id, \qquad A,B \in \mathbb{R}^{n \times n}, \quad c,d \in \mathbb{R}^n
\end{equation}
for $x,y \in \mathbb{R}^n$  with Frobenius inversion in a way that does not require computing explicit inverse and the circumstances under which this method is superior. 
For the sake of discussion, we will assume throughout this section that we use LU factorization, computed using Gaussian Elimination with Partial Pivoting, as our main tool, but one may easily substitute it with any other standard matrix decomposition.

The most straightforward way to solve \eqref{eq:linear system} would be directly as a complex linear system with coefficient matrix $ A + iB \in \mathbb{C}^{n \times n}$. As we mentioned at the beginning of this article, the IEEE-754 floating point standard \cite{ieee} does not support complex floating point arithmetic and relies on software to convert them to real floating point arithmetic \cite[p.~55]{Overton}. For greater control, we might instead transform \eqref{eq:linear system} into a real linear system with coefficient matrix $\begin{bsmallmatrix} A & -B \\ B & A \end{bsmallmatrix} \in \mathbb{R}^{2n \times 2n}$. Note that the condition numbers of the coefficient matrices are identical:
\[
\kappa_2(A + iB) = \kappa_2 \biggl( \begin{bmatrix} A & -B \\ B & A \end{bmatrix} \biggr).
\]
However, an alternative that takes advantage of Frobenius inversion \eqref{eq:Finv} would be Algorithm~\ref{alg:Frob:linear}. For simplicity, we assume that $A$ is invertible below but if not, this can be easily addressed with a simple trick in Section~\ref{sec:rand}.

\begin{algorithm}[!htb]
\caption{Linear system with Frobenius inversion and LU factorization}
\label{alg:Frob:linear}
\begin{algorithmic}[1]
\Require $A+iB \in \GL_n(\mathbb{C})$ with $A \in \GL_n(\mathbb{R})$, $c,d \in \mathbb{R}^n$
\State LU factorize $A = P_1^\tp L_1 U_1$;
\State forward and backward substitute for $X_1$ in $L_1 U_1 X_1 = P_1 B$;
\State matrix multiply and add $X_2 = A+BX_1$;
\State LU factorize $X_2 = P_2^\tp L_2 U_2$;
\State forward and backward substitute for $x_1,y_1$ in $L_2 U_2 [x_1,y_1] = P_2 [c,d]$;
\State forward and backward substitute for $x_2,y_2$ in $L_1 U_1 [x_2,y_2] = P_1 B[y_1,x_1]$;
\State vector add $x =  x_1 + x_2 $, $y = y_2-y_1$;
\Ensure solution of $(A+iB)(x+iy) = c+id$
\end{algorithmic}
\end{algorithm}

One may easily verify that Algorithm~\ref{alg:Frob:linear} gives the solution as claimed, by virtue of the expression \eqref{eq:Finv} for Frobenius inversion.
Observe that Algorithm~\ref{alg:Frob:linear} involves only the matrices $A$, $B$, and $A +BA^{-1}B$, unsurprising since these are the matrices that appear in \eqref{eq:Finv}. In the rest of this section, we will establish that there is an open subset of matrices $A + iB \in \mathbb{C}^{n \times n}$ with
\begin{equation}\label{eq:cond1}
\max\bigl( \kappa_2(A), \kappa_2(B), \kappa_2 (A + BA^{-1}B) \bigr) \ll \kappa_2(A + iB).
\end{equation}
A consequence is that ill-conditioned complex matrices with well-conditioned real and imaginary parts are common; in particular, there are uncountably many and they occur with positive probability with respect to any reasonable probability measure (e.g., Gaussian) on  $\mathbb{C}^{n \times n}$.
In fact we will show in Theorems~\ref{thm:illcond1} and \ref{thm:illcond1} that $A + iB  \in \mathbb{C}^{n \times n}$ or  $\begin{bsmallmatrix} A & -B \\ B & A \end{bsmallmatrix} \in \mathbb{R}^{2n \times 2n}$  can be arbitrarily ill-conditioned when $A$ and $B$ are well-conditioned or even perfectly conditioned, a situation that is tailor-made for  Algorithm~\ref{alg:Frob:linear}.

\begin{lemma}\label{lem:condnumber}
Let $A,B \in \mathbb{R}^{n\times n}$ with $ A  = G H$ for some $G,H\in \GL_n(\mathbb{R})$. Let $N \coloneqq H^{-1} B G^{-1}$. Then
\[
\kappa(G) \kappa(I + iN) \kappa (H) \ge \kappa(A + iB)  \ge \max\left\lbrace \frac{\lVert (I + iN)^{-1}  \rVert}{\kappa(H)}, \frac{\lVert (I + iN)^{-1}  \rVert}{\kappa(G)}
\right\rbrace.
\]
\end{lemma}
\begin{proof}
Let $X \coloneqq A + iB = H (I + i N) G$. Then
\[
\kappa(X) \le \kappa(H) \kappa(I + iN) \kappa (G).
\]
For the other inequality, since $X^{-1} = G^{-1} (I + i N)^{-1} H^{-1}$,
\begin{equation}\label{eq:condnumber0}
\kappa(X) = \lVert X \rVert \lVert X^{-1} \rVert = \lVert A + iB \rVert  \lVert G^{-1} (I + i N)^{-1} H^{-1} \rVert.
\end{equation}
As $\lVert X v \rVert =  \lVert A v + i B v \rVert \ge \lVert A v \rVert$ for all $v \in \mathbb{R}^n$, we have
\begin{equation}\label{eq:condnumber1}
\lVert X \rVert \ge \lVert A \rVert.
\end{equation}
Since the spectral norm is submultiplicative,
\[
\lVert X  \rVert = \lVert G G^{-1} X H^{-1} H  \rVert \le \lVert G \rVert  \lVert G^{-1} X H^{-1} \rVert  \lVert H \rVert, \qquad
\lVert H \rVert = \lVert A G^{-1} \rVert \le \lVert A \rVert \lVert G^{-1} \rVert,
\]
and we obtain
\begin{equation}\label{eq:condnumber2}
\lVert G^{-1} X H^{-1} \rVert \ge \frac{\lVert X  \rVert}{\lVert G \rVert \lVert H \rVert} \ge \frac{\lVert X  \rVert}{\lVert G \rVert \lVert G^{-1} \rVert \lVert A \rVert} = \frac{\lVert X  \rVert}{\kappa(G) \lVert A \rVert}.
\end{equation}
Assembling \eqref{eq:condnumber0}--\eqref{eq:condnumber2}, we get
\[
\kappa(X) \ge \lVert A \rVert  \frac{\lVert (I + iN)^{-1}  \rVert}{\kappa(G) \lVert A \rVert} =  \frac{\lVert (I + iN)^{-1}  \rVert}{\kappa(G)}. 
\]
Swapping the roles of $G$ and $H$, we get $\kappa(X)  \ge \lVert (I + iN)^{-1}  \rVert/\kappa(H)$.
\end{proof}

Choosing specific matrix decompositions $A = G H$ and imposing conditions on $N$ in Lemma~\ref{lem:condnumber} allows us to deduce better bounds for $\kappa(A + iB)$.
\begin{corollary}\label{cor:condnumber}
Let $A \in \GL_n(\mathbb{R})$ and $B \in \mathbb{R}^{n\times n}$. Let $A = Q R$ be a QR decomposition with $Q\in \O_n(\mathbb{R})$ and $R \in \mathbb{R}^{n \times n}$ upper triangular. If $N = Q^\tp B R^{-1}$ is a normal matrix with eigenvalues $\lambda_1,\dots, \lambda_n \in \mathbb{C}$, then
\[
\kappa(A) \frac{\max_{k = 1,\dots,n}\lvert 1 + i \lambda_k \rvert }{\min_{k = 1,\dots,n}\lvert 1 + i \lambda_k \rvert}  \ge \kappa(A + iB)  \ge \max_{k = 1,\dots,n}  \frac{1}{ \lvert 1 + i \lambda_k \rvert }.
\]
\end{corollary}
\begin{proof}
It suffices to observe that $\kappa (A) = \kappa (R)$, $\lVert A \rVert = \lVert R \rVert$, and $N$ is unitarily diagonalizable.
\end{proof}

We may now show that for any well-conditioned $A \in \mathbb{R}^{n\times n}$, there is a well-conditioned $B \in \mathbb{R}^{n\times n}$ such that $A + i B \in \mathbb{C}^{n \times n}$ is arbitrarily ill-conditioned (i.e., $\gamma \to \infty$).
\begin{theorem}[Ill-conditioned matrices with well-conditioned real and imaginary parts]\label{thm:illcond1}
Let $A \in \GL_n(\mathbb{R})$ and $\gamma  \ge 1$. Then there exists $B \in \mathbb{R}^{n\times n}$ such that 
\[
\kappa(A)  \ge \kappa(B),\qquad \kappa(A + iB) \ge \gamma.
\]
\end{theorem}
\begin{proof}
Consider the normal matrix 
\[
N = \begin{bmatrix}
0 & -t & 0 & \cdots & 0 \\
t & 0 & 0 & \cdots & 0 \\
0 & 0 & t & \cdots & 0 \\ 
\vdots & \vdots & \vdots & \ddots & \vdots \\
0 & 0 & 0 & \cdots & t
\end{bmatrix} \in \mathbb{R}^{n \times n}
\]
where $t\ge 0$ is a  real parameter to be chosen later. The eigenvalues of $N$ are $\pm it$ and $t$ so $\kappa(N) = 1$. Let $A = QR$ be the QR decomposition of $A$ and set $B \coloneqq QNR$. Then $\kappa (B) \le \kappa(N) \kappa(A) = \kappa(A)$. By Corollary~\ref{cor:condnumber}, we have $\kappa (A + iB)  \ge 1/(1 -t)$. Hence if $t$ is chosen in the interval $[1 -1/\gamma , 1)$, we get $\kappa(A + iB) \ge  \gamma$.
\end{proof}

Interestingly, we may use the Frobenius inversion formula to push  Theorem~\ref{thm:illcond1} to the extreme, constructing an arbitrarily ill-conditioned complex matrix with perfectly conditioned real and imaginary parts.
\begin{proposition}
Let $\gamma \ge 1$. There exists $A + iB \in \mathbb{C}^{n \times n}$ with $\kappa (A + iB) \ge \gamma$ and $\kappa(A) = \kappa(B) = 1$.
\end{proposition}
\begin{proof}
Let $Q\in \O_n(\mathbb{R})$. The Frobenius inversion formula \eqref{eq:Finv} gives
\[
( I + i Q)^{-1} = Q^\tp (Q^\tp + Q)^{-1} - i (Q+Q^\tp)^{-1} = (Q^\tp - i I) (Q + Q^\tp)^{-1}  = (I - i Q) (Q^2 + I)^{-1}.
\]
An orthogonal matrix must have an eigenvalue decomposition of the form $Q = U \Lambda U^\ha $ for some $U\in \U_n(\mathbb{C})$ and $\Lambda =\diag(\lambda_1,\dots,\lambda_n)\in \mathbb{C}^{n\times n}$ with  $\lvert \lambda_1\rvert = \dots = \lvert \lambda_n \rvert = 1$. Therefore
$I + iQ = U\diag( 1 + i \lambda_1, \dots,  1 + i \lambda_n) U^\ha $ and
\[
(I - i Q) (Q ^2+ I)^{-1} = U \diag \left( \frac{1 - i \lambda_1}{ 1 + \lambda_1^2}, \dots, 
\frac{1 - i \lambda_n }{ 1 + \lambda_n^2}
\right) U^\ha  = U \diag \left( \frac{1}{ 1 + i \lambda_1}, \dots, 
\frac{1}{ 1 + i \lambda_n}
\right) U^\ha  .
\]
Since the spectral norm is unitarily invariant,
\[
\lVert I + i Q \rVert = \max_{k = 1,\dots,n}  \; \lvert  1 + i\lambda_k\rvert,\qquad 
\lVert (I - i Q) (Q ^2+ I)^{-1} \rVert = \max_{k = 1,\dots,n}\; \biggl\lvert  \frac{1 - i \lambda_k }{ 1 + \lambda_k^2} \biggr\rvert,
\]
from which we deduce that 
\[
\kappa(I + iQ) = \Bigl(\max_{k = 1,\dots,n}  \; \lvert  1 + i\lambda_k \rvert \Bigr) \cdot
\biggl( \max_{k = 1,\dots,n}\; \biggl\lvert  \frac{1 - i \lambda_k }{ 1 + \lambda_k^2} \biggr\rvert \biggr) \ge \sqrt{2} \max_{k = 1,\dots,n}
\biggl\lvert  \frac{1}{ 1 + i \lambda_k} \biggr\rvert.
\]
The last inequality follows from the fact that $\lambda_k $'s are unit complex numbers. Now observe that if we choose $\lambda_k$ to be sufficiently close to $i$, then $\kappa(I + iQ)$ can be made arbitrarily large. 
\end{proof}

Note that Frobenius inversion \eqref{eq:Finv} and Algorithm~\ref{alg:Frob:linear} avoids $A + iB$ and work instead with the matrices $A$, $B$, and $A + BA^{-1}B$. It is not difficult to tweak Theorem~\ref{thm:illcond1} to add $A + BA^{-1}B$  to the mix.
\begin{theorem}[Ill-conditioned matrices for Frobenius inversion]\label{thm:illcond2}
Let $A \in \GL_{2n}(\mathbb{R})$ and $\gamma  \ge 1$. Then there exists $B \in \mathbb{R}^{2n\times 2n}$ such that 
\[
\kappa(A)  \ge \max \bigl( \kappa(B),\kappa(A + BA^{-1}B)\bigr),\qquad \kappa(A + iB) \ge \gamma.
\]
In other words, for any invertible matrix $A$ there exists a well-conditioned $B$ such that $A + BA^{-1}B$ is well-conditioned but $A + i B$ is arbitrarily ill-conditioned.
\end{theorem}
\begin{proof}
Consider the skew-symmetric (and therefore normal) matrix 
\[
N = \begin{bmatrix}
0 & -t  & \cdots & 0 & 0 \\
t & 0  & \cdots & 0  & 0\\
\vdots & \vdots & \ddots & \vdots &\vdots \\
0 & 0  & \cdots & 0 & -t \\ 
0 & 0  & \cdots & t & 0 
\end{bmatrix}  \in \mathbb{R}^{2n \times 2n},
\]
where $t\ge 0$ is a  real parameter to be chosen later. The eigenvalues of $N$ are $\pm it$ so $\kappa(N) = 1$. Let $A = QR$ be the QR decomposition of $A$ and set $B \coloneqq QNR$. Then $\kappa (B) \le \kappa(N) \kappa(A) = \kappa(A)$. By Corollary~\ref{cor:condnumber}, we have $\kappa (A + iB)  \ge 1/(1 -t)$. Hence if $t$ is chosen in the interval $[1 -1/\gamma , 1)$, we get $\kappa(A + iB) \ge  \gamma$. We also have 
\[
A + BA^{-1}B = QR + (Q N R) (R^{-1}Q^\tp) (Q N R) = Q(I+ N^2)R
\]
and as $I + N^2 = (1-t^2) I$, we see that $\kappa(I+ N^2) = 1$ and so
\[
\kappa(A + BA^{-1}B) \le \kappa(I + N^2) \kappa(A) \le \kappa(A). \qedhere
\]
\end{proof}

Theorems~\ref{thm:illcond1} and \ref{thm:illcond2} show the existence of arbitrarily ill-conditioned complex matrices with well-conditioned real and imaginary parts (and also  $A + BA^{-1}B$ in the case of Theorem~\ref{thm:illcond2}). We next show that such matrices exist in abundance --- not only are there uncountably many of them, they occur with nonzero probability, showing that there is no shortage of matrices where Algorithm~\ref{alg:Frob:linear} provides an edge by avoiding the ill-conditioning of $A + iB $ or, equivalently, of $\begin{bsmallmatrix} A & -B \\ B & A \end{bsmallmatrix}$.
\begin{proposition}\label{prop:illcond}
Let $\mathcal{S}_n \coloneqq \{A + iB \in \GL_n(\mathbb{C}): A,B \in \GL_n(\mathbb{R})\}$.  For any $1 < \beta \le \gamma < \infty$,
\begin{align*}
\bigl\{A + iB \in  \mathcal{S}_n  &: \kappa(B) \le \kappa(A) \le \beta,\; \kappa(A + iB) \ge \gamma \bigr\},\\
\bigl\{A + iB \in  \mathcal{S}_{2n}  &: \max\bigl( \kappa(B) \kappa(A + BA^{-1}B)\bigr) \le \kappa(A) \le \beta,\; \kappa(A + iB) \ge \gamma \bigr\}
\end{align*}
have nonempty interiors in $\mathbb{C}^{n \times n}$ and  $\mathbb{C}^{2n \times 2n}$ respectively.
\end{proposition}
\begin{proof}
Consider the maps $\varphi_1: \mathcal{S}_n \to [1,\infty) \times  \mathbb{R} \times [1,\infty)$ and $\varphi_2: \mathcal{S}_{2n} \to [1, \infty] \times \mathbb{R}  \times \mathbb{R}  \times [1,\infty)$ defined by
\begin{align*}
\varphi_1(A + iB) &= \bigl(\kappa(A), \kappa(A) - \kappa(B), \kappa(A + iB) \bigr), \\
\varphi_2(A + iB) &=  \bigl(\kappa(A), \kappa(A) - \kappa(B), \kappa(A) - \kappa(A + B A^{-1} B), \kappa(A + iB) \bigr)
\end{align*}
respectively. These are continuous since the condition number $\kappa$ is a continuous function on invertible matrices. For any $\gamma \ge \beta > 1$, the preimage  $\varphi_1^{-1}( [1, \beta] \times [0,\infty) \times [\gamma,\infty)) \ne \varnothing$ by Theorem~\ref{thm:illcond1} and $\varphi_2^{-1} ( (1, \beta] \times [0,\infty) \times [0,\infty) \times [\gamma,\infty) ) \ne \varnothing$ by Theorem~\ref{thm:illcond2}. Note that these preimages are precisely the required sets in question and by continuity of $\varphi_1$ and $\varphi_2$ they must  have nonempty interiors.
\end{proof} 

One may wonder if there is a flip side to Theorems~\ref{thm:illcond1} and \ref{thm:illcond2}, i.e., are there complex matrices whose condition numbers are controlled by their real and imaginary parts? We conclude this section by giving a construction of such matrices.
\begin{proposition}
Let $A, B\in \GL_n(\mathbb{R})$. If $\sigma_n(A) = \mu \sigma_1(B) $ for some $\mu > 1$, then 
\[
\frac{\kappa(A) - 1}{2}< \kappa(A + i B) \le \kappa(A) + \frac{\kappa(A) + 1}{\mu - 1}
\]  
In particular, if $A$ is well-conditioned and $\mu \gg 1$, then $A + iB$ is also well-conditioned.
\end{proposition}
\begin{proof}
We first show a more generally inequality that holds for arbitrary $X, Y\in \mathbb{C}^{n\times n}$. Recall that singular values satisfy
\[
\sigma_{i + j - 1}(X + Y) \le \sigma_{i}(X) + \sigma_{j}(Y),\quad 1\le  i + j - 1 \le n.
\]
In particular, we have  $\sigma_1(X + Y) \le \sigma_1(X) + \sigma_1(Y)$ and $\sigma_1((X+Y) + (-Y) ) \le \sigma_1(X+Y) + \sigma_1(Y)$, and therefore
\[
\sigma_1(X) - \sigma_1(Y) \le \sigma_1(X+Y)  \le \sigma_1(X) + \sigma_1(Y).
\]
Also, we have  $\sigma_n(X + Y) \le \sigma_n(X) + \sigma_1(Y)$ and $\sigma_n((X+Y) + (-Y) ) \le \sigma_n(X+Y) + \sigma_1(Y)$, and therefore
\[
\sigma_n(X) - \sigma_1(Y) \le \sigma_n(X+Y)  \le \sigma_n(X) + \sigma_1(Y).
\]
If $\sigma_n(X) > \sigma_1(Y)$, then 
\[
\frac{\sigma_1(X) - \sigma_1(Y)}{\sigma_n(X) + \sigma_1(Y)} \le \frac{\sigma_1(X+Y)}{\sigma_n(X+Y)} \le \frac{\sigma_1(X) + \sigma_1(Y)}{\sigma_n(X) - \sigma_1(Y)}.
\]
Rewriting in terms of condition number,
\[
\frac{\kappa(X) \sigma_n(X) - \sigma_1(Y)}{\sigma_n(X) + \sigma_1(Y)}  \le \kappa (X+Y) \le \frac{\kappa(X) \sigma_n(X) + \sigma_1(Y)}{\sigma_n(X) - \sigma_1(Y)}.
\]
Hence
\[
\kappa(X) - (\kappa(X)+1) \left[ \frac{\sigma_1(Y)}{\sigma_n(X) + \sigma_1(Y)} \right] \le \kappa (X+Y) \le \kappa(X) + (1 + \kappa(X)) \left[ \frac{\sigma_1(Y)}{\sigma_n(X) - \sigma_1(Y)} \right].
\]
Since $\sigma_n(X) > \sigma_1(Y)$, we have
\[
\frac{\sigma_1(Y)}{\sigma_n(X) + \sigma_1(Y)}  < \frac{1}{2}
\] 
and so $\kappa(X+Y) > (\kappa(X) - 1)/2$.  If we set $X = A$, $Y = iB$, and substitute  $\sigma_n(A) = \mu \sigma_1(B) $, the required inequality follows.
\end{proof}

\section{Computing explicit inverse with Frobenius inversion}\label{sec:explicit}

We have discussed at length in Section~\ref{sec:yes} why computing an explicit inverse for a matrix is sometimes an inevitable or desirable endeavor. Here we will discuss the numerical properties of inverting a complex matrix using Frobenius inversion. The quadratic extension $\mathbb{C}$ over $\mathbb{R}$ falls under Algorithm~\ref{alg:inverse} (as opposed to Algorithm~\ref{alg:inverse2}) and here we will compare its computational complexity with the complex matrix inversion algorithm based on LU decomposition, the standard method of choice for computing explicit inverse in \textsc{Matlab}, Maple, Julia, and Python.

\begin{algorithm}[!htb]
\caption{Inversion with LU decomposition}
\label{alg:LU1}
\begin{algorithmic}[1]
\Require $X \in \GL_n(\mathbb{C})$
\State LU factorize $X = P^\tp LU$;
\State backward substitute for $X_1$ in  $UX_1 = I$; \label{alg:LU1:step2}
\State forward substitute for $X_2$ in $X_2L =X_1$; \label{alg:LU1:step3}
\Ensure inverse $X^{-1} = X_2P$
\end{algorithmic}
\end{algorithm}

Strictly speaking, Algorithm~\ref{alg:LU1} computes the left inverse of the input matrix $X$, i.e., $YX = I$. We may also compute its right inverse, i.e., $XY = I$, by swapping the order of backward and forward substitutions.  Even though the left and right inverse of a matrix are always equal mathematically, i.e., in exact arithmetic, they can be different numerically, i.e., in floating-point arithmetic \cite{HighamBook}. Any subsequent mentions of Algorithm~\ref{alg:LU1} would also hold with its right inverse variant.

\subsection{Floating point complexity}\label{sec:compare}

In Section~\ref{sec:symb}, we discussed computational complexity of Frobenius inversion for $\inv_{n,\mathbb{F}}$ in units of $\inv_{n,\Bbbk}$, $\mul_{n,\Bbbk}$, $\add_{n,\Bbbk}$, which are in turn treated as black boxes. Here, for the case of $\mathbb{F} = \mathbb{C}$ and $\Bbbk = \mathbb{R}$, we will count actual real flops, i.e., real floating point operations; we will not distinguish between the cost of real addition and real multiplication since there is no noticeable difference in their latency on modern processors --- each would count as a single flop. With this in mind, we do not use Gauss multiplication for complex \emph{numbers} since it trades one real multiplication for three real additions, i.e., more expensive if real addition costs the same as real multiplication. We caution our reader that this says nothing about Gauss multiplication for complex \emph{matrices} since real matrix addition ($n^2$ flops) is still much cheaper than real matrix multiplication ($2n^3$ flops).

Our implementation of Frobenius inversion in Algorithm~\ref{alg:inverse} requires real matrix multiplication and real matrix inversion as subroutines. Let   $\mathscr{A}_{\inv}$ and $\mathscr{A}_{\mul}$ be respectively any two algorithms for real matrix inversion and real matrix multiplication, with real flop counts $T_{\inv}(n)$ and $T_{\mul}(n)$ on real $n \times n$ matrix inputs. There is little loss of generality in making two mild assumptions about the inversion algorithm $\mathscr{A}_{\inv}$:
\begin{enumerate}[\normalfont(i)]
\item\label{it:c2r} $\mathscr{A}_{\inv}$ also works for complex matrix inputs at the cost of a multiple of $T_{\inv}(n)$, the multiple being the cost of a complex flop in terms of real flops;

\item\label{it:a=m} the number of complex additions and the number of complex multiplications in $\mathscr{A}_{\inv}$ applied to complex matrix inputs both take the form $cn^k + $ lower order terms, i.e., same dominant term but lower order terms may differ.
\end{enumerate}
Note that these assumptions are satisfied if $\mathscr{A}_{\inv}$ is chosen to be Algorithm~\ref{alg:LU1}, even if we replace the LU decomposition in  them by other decompositions like QR or Cholesky (if applicable).

\begin{theorem}[Frobenius inversion versus standard inversion]\label{thm:threshold}
Let $\lambda > 0$ be such that the cost of computing $A^{-1} B$ for any $A\in \GL_n(\mathbb{R})$, $B \in \mathbb{R}^{n \times n}$ is bounded by $\lambda T_{\mul}(n)$.  Algorithm~\ref{alg:inverse} with subroutines $\mathscr{A}_{\inv}$ and $\mathscr{A}_{\mul}$ on real inputs $A$ and $B$ is asymptotically faster than directly applying $\mathscr{A}_{\inv}$ on complex input $A +iB$ if and only if
\[
\lim_{n\to \infty}  \frac{ T_{\inv} (n)}{T_{\mul} (n)} >  \frac{2 + \lambda}{3}.  \qedhere
\]
\end{theorem}
\begin{proof}
The first two steps of Algorithm~\ref{alg:inverse} computes $A^{-1} B$, which costs $\lambda T_{\mul}(n)$ operations. Thereafter, computing $BA^{-1}B$ costs one matrix multiplication, $A + BA^{-1}B$ one matrix addition, $S = (A + BA^{-1}B)^{-1}$ one matrix inversion, and finally $A^{-1}B S$ one matrix multiplication. We disregard matrix addition since it takes $O(n^2)$ flops and does not contribute to the dominant term. So the cost in real flops of Algorithm~\ref{alg:inverse} is dominated by $T_{\inv}(n) + (2 + \lambda)T_{\mul}(n)$ for $n$ sufficiently large.

Now suppose we apply $\mathscr{A}_{\inv}$ directly to the complex matrix $A + i B$. Each complex addition costs two real flops (real additions) and each complex multiplication costs six real flops (four real multiplications and two real additions). Also, by assumption~\ref{it:a=m}, $\mathscr{A}_{\inv}$ has the same number of real additions and real multiplications, ignoring lower order terms.  So the cost in real flops of $\mathscr{A}_{\inv}$ applied directly to $A + iB \in \mathbb{C}^{n \times n}$ is dominated by $4T_{\inv}(n)$ for $n$ sufficiently large, i.e., the `multiple' in assumption~\ref{it:c2r} is $4$.

Hence Algorithm~\ref{alg:inverse} is faster than $\mathscr{A}_{\inv}$ if and only if
\[
    4T_{\inv}(n) > T_{\inv}(n) +(2 + \lambda )T_{\mul}(n)
\]
for $n$ sufficiently large, i.e., $\lim_{n\to \infty}  T_{\inv}(n)/T_{\mul}(n)  > (2 + \lambda)/3$. 
\end{proof}

As we discussed in Section~\ref{sec:Frob inv}, Algorithm~\ref{alg:inverse} is written in a general form that applies over arbitrary fields and to both symbolic and numerical computing. However, when restricted to $\Bbbk = \mathbb{R}$, $\mathbb{F} = \mathbb{C}$, and with numerical computing in mind, we may state a more specific version Algorithm~\ref{alg:inverse3} involving LU decomposition and backsubstitutions for $AX = B$. 

\begin{algorithm}[!htb]
\caption{Frobenius inversion with LU decomposition}
\label{alg:inverse3}
\begin{algorithmic}[1]
\Require $X = A + i B \in \GL_n(\mathbb{C})$ with $A \in \GL_n(\mathbb{R})$ 
\State LU factorize $A = P^\tp_1 L_1U_1$; \label{step:LU1}
\State forward and backward substitute for $X_1$ in $L_1 U_1 X_1 = P_1 B$; \label{step:back1}
\State matrix multiply and add $X_ 2 = A + BX_1$; 
\State LU factorize $X_2 = P^\tp_2 L_2U_2$; \label{step:LU2}
\State forward and backward substitute for $X_3,X_4$ in $[X_3, X_4] P_2 L_2 U_2 = [I, X_1]$; \label{step:back2}
\Ensure inverse $X^{-1} = X_3 - i X_4$
\end{algorithmic}
\end{algorithm}

Note that Steps~\ref{step:LU1} and \ref{step:back1} in Algorithm~\ref{alg:inverse3} are essentially just Algorithm~\ref{alg:LU1} with a different right-hand side, and likewise for Steps~\ref{step:LU2} and \ref{step:back2}. Hence we may regard Algorithm~\ref{alg:inverse3} as Algorithm~\ref{alg:inverse} with $\mathscr{A}_{\inv}$ given by Algorithm~\ref{alg:LU1} and $\mathscr{A}_{\mul}$ given by standard matrix multiplication. These choices allow us to assign flop counts to illustrate Theorem~\ref{thm:threshold}.

\begin{proposition}[Flop counts]\label{prop:threshold}
Algorithm~\ref{alg:inverse3} has a real flop count of $28n^3/3$ whereas applying Algorithm~\ref{alg:LU1} directly to a complex matrix has a real flop count of $32n^3/3$.
\end{proposition}
\begin{proof}
These come from a straightforward flop count of the respective algorithms, dropping lower order terms.
\end{proof}
With these choices for $\mathscr{A}_{\inv}$ and $\mathscr{A}_{\mul}$, the cost of computing $A^{-1} B$ is asymptotically bounded by $\frac{4}{3} T_{\mul}(n)$ --- one LU decomposition plus $2n$ forward and backward substitutions. So $\lambda = 4/3$ and $(2 + \lambda)/3 = 10/9 < 4/3 = \lim_{n\to \infty}  T_{\inv} (n)/T_{\mul}(n)$. Hence inverting a complex matrix via Algorithm~\ref{alg:inverse3} is indeed faster than inverting it directly with Algorithm~\ref{alg:LU1}, as predicted by Theorem~\ref{thm:threshold}; we will also present numerical evidence that supports this in Section~\ref{sec:expr}.

Proposition~\ref{prop:threshold} also tells us that variations of Frobenius inversion formula like the one proposed in \cite{inversion_ref9} can obliterate the computational savings afforded by Frobenius inversion. These variants all take the form $X^{-1} = (ZX)^{-1} Z$ for some $Z\in \GL_n(\mathbb{C})$ and the extra matrix multiplications incur additional costs. As a result, the variant in \cite{inversion_ref9} takes $34n^3$ real flops, which exceeds even the $32n^3/3$ by standard methods (e.g., Algorithm~\ref{alg:LU1}).

\subsection{Almost sure Frobenius inversion}\label{sec:rand}

One obvious shortcoming of Frobenius inversion is that \eqref{eq:Finv}  requires the real part $A$ to be invertible. It is easy to modify \eqref{eq:Finv}  to
\[
(A + i B)^{-1} =  B^{-1}A(AB^{-1}A + B)^{-1} - i (AB^{-1}A + B)^{-1}
\]
if  $B$ is invertible. Nevertheless we may well have circumstances where $A +iB$ is invertible but neither $A$ nor $B$ is, e.g., $\begin{bsmallmatrix} 1 & 0 \\ 0 & i \end{bsmallmatrix}$. Here we will extend Frobenius inversion to any invertible $A + iB$ in a way that preserves its computational complexity --- this last qualification is important. As we saw in Proposition~\ref{prop:threshold}, the speed improvement of Frobenius inversion comes from the constants, i.e., it inverts an $n \times n$ complex matrix with $28n^3/3$ real flops whereas Algorithm~\ref{alg:LU1} takes $32n^3/3$ real flops. As we noted in Section~\ref{sec:prev}, prior attempts such as those in \cite{inversion_ref9,1100887} at extending Frobenius inversion to all $A + iB \in \GL_n(\mathbb{C})$ invariably require the inversion of a real matrix of size $2n \times 2n$, thereby obliterating any speed improvement afforded by Frobenius inversion.

Our approach avoids any matrices of larger dimension by adding a simple randomization step that in turns depend on the following observation.
\begin{lemma}\label{lem:rand Frob}
Let $A + i B \in \GL_n(\mathbb{C})$ with $A,B\in \mathbb{R}^{n\times n}$. Then there exist at most $n$ values of $\mu \in \mathbb{R}$ such that $A - \mu B$ is singular.
\end{lemma}
\begin{proof}
As $f(t) \coloneqq \det (A + t B)$ is a polynomial of degree at most $n$ and $f(i) = \det(A + iB) \ne 0$, $f$ has at most $n$ zeros in $\mathbb{C}$. So $A-\mu B$ is invertible for all but at most $n$ values of $\mu \in \mathbb{C} \supseteq \mathbb{R}$.
\end{proof}

Algorithm~\ref{alg:rand inverse} is essentially Frobenius inversion applied to  $(1 + \mu i) (A + iB)$ for some random $\mu$.  Note that $A - \mu B$ is exactly the real part of $(1 + \mu i) (A + iB)$ and therefore invertible for all but at most $n$ values of $\mu$ by Lemma~\ref{lem:rand Frob}. The interval $(0,1)$ is chosen so that we may generate $\mu$ from the uniform distribution  but we could have also used $\mathbb{R}$ with standard normal distribution, both of which are standard implementations in numerical packages.

\begin{algorithm}[!htb]
\caption{Almost sure Frobenius inversion}
\label{alg:rand inverse}
\begin{algorithmic}[1]
\Require $X = A + i B \in \GL_n(\mathbb{C})$
\State randomly generate $\mu \in [0, 1]$; \label{alg:rand inverse:1}
\State matrix add $X_1 = A - \mu B$, $X_2 = \mu A + B$;\label{alg:rand inverse:2}
\State Frobenius invert $X_3 + i X_4 = (X_1  + i X_2 )^{-1}$; \Comment {calls Algorithm~\ref{alg:inverse3}} \label{alg:rand inverse:3}
\State matrix add $X_5 = X_3 - \mu X_4$, $X_6 = \mu X_3 +  X_4$; \label{alg:rand inverse:5}
\Ensure inverse $X^{-1} = X_5 + i X_6$
\end{algorithmic}
\end{algorithm}

Note that Algorithm~\ref{alg:inverse3} fails on a set of real dimension $2n^2 -1$, i.e., when the real part of the input is singular, but Algorithm~\ref{alg:rand inverse} has reduced this to a set of dimension zero.

\begin{proposition}\label{prop: rand inverse}
Algorithm~\ref{alg:rand inverse} has the same asymptotic time complexity as that of Algorithm~\ref{alg:inverse3}, i.e., Frobenius inversion. Algorithm~\ref{alg:rand inverse} works with probability one if $\mu$ is chosen  randomly from $[0,1]$ with any non-atomic probability measure.
\end{proposition}
\begin{proof}
The time complexity of Algorithm~\ref{alg:rand inverse} is that of  Algorithm~\ref{alg:inverse3} plus the matrix additions in Steps~\ref{alg:rand inverse:2} and \ref{alg:rand inverse:5} that cost a total of $4 \times 2n^2$ real flops. By Proposition~\ref{prop:threshold}, the time complexity of Algorithm~\ref{alg:rand inverse} is dominated by $28n^3/3$, and therefore the lower order term $8n^2$ may be ignored asymptotically.

By Lemma~\ref{lem:rand Frob}, $X_1 = A - \mu B$ in Step~\ref{alg:rand inverse:2} is invertible with probability one since any finite subset of $[0,1]$ is a null set with a non-atomic probability measure. Thus Algorithm~\ref{alg:inverse3} is applicable to $X_1 + i X_2$ and we have
\[
(X_3 - \mu X_4) + i (\mu X_3 +  X_4) = (1 + \mu i)  (X_1 + iX_2)^{-1} = (1 + \mu i) ]\bigl( X (1 + \mu i) \bigr)^{-1} = X^{-1}.
\]
The almost sure correctness of Algorithm~\ref{alg:rand inverse} follows.
\end{proof}

\subsection{Hermitian positive definite matrices}\label{sec:herm}

The case of Hermitian positive definite $A + iB$ deserves special consideration given their ubiquity. We will propose and analyze a new variant of Frobenius inversion that exploits this special structure of  $A + iB$. The happy coincidence is that a Hermitian positive definite $A + iB \in \mathbb{C}^{n \times n}$ must necessarily have symmetric positive definite $A$ and $A + BA^{-1}B \in \mathbb{R}^{n \times n}$ as well as a skew-symmetric $B \in \mathbb{R}^{n \times n}$ --- precisely the matrices we need in Frobenius inversion. Various required quantities may thus be computed via Cholesky decompositions $A = R_1^\tp  R_1$  and $A + BA^{-1}B = R_2^\tp R_2$:
\begin{equation}\label{eq:cholesky}
\begin{aligned}
A^{-1} B  &= R_1^{-1} R_1^{-\tp} B, \\
BA^{-1}B  &= B R_1^{-1} R_1^{-\tp} B = -(R_1^{-\tp}B)^\tp  (R_1^{-\tp} B),\\
(A + BA^{-1}B)^{-1}  &=  \bigl( A -(R_1^{-\tp}B)^\tp (R_1^{-\tp} B) \bigr)^{-1} = R_2^{-1} R_2^{-\tp}, \\
A^{-1}B (A + BA^{-1}B)^{-1} &= A^{-1}B R_2^{-1} R_2^{-\tp}.
\end{aligned}
\end{equation}
\begin{lemma}\label{lem:positive real part}
Let $A + iB \in \mathbb{C}^{n \times n}$ be a Hermitian positive definite matrix with $A,B\in \mathbb{R}^{n\times n}$. Then
\begin{enumerate}[\normalfont(i)]
\item $A$ is symmetric positive definite and $B$ is skew-symmetric; \label{lem:item-1}
\item $A + BA^{-1}B$ is symmetric positive definite. \label{lem:item-2}    
\end{enumerate}
In particular, $A$ is always invertible and so there is no need for an analogue of Algorithm~\ref{alg:rand inverse}.
\end{lemma}
\begin{proof}
Let $X = A + iB $ and write $X \succ 0$ to indicate positive definiteness. Then since $A = (X + \overline{X})/2$ and $B = (X - \overline{X})/2i$, $A$ is symmetric and $B$ is skew-symmetric given that $X$ is Hermitian. Since $X \succ 0$, for any $x\in \mathbb{R}^n$,
\[
x^\tp A x =\frac12 x^\ha   (X + \overline{X}) x = \frac12 x^\ha  X x + \frac12 \overline{x^\ha  X x} = x^\ha  X x \ge 0,
\]
with equality if and only if $x = 0$, showing that $A$ is positive definite. Again since $X \succ 0$ , for any $z\in \mathbb{C}^n$,
\[
z^\ha  \overline{X} z = \overline{\overline{z}^\ha  X \overline{z}} \ge 0,
\]
with equality if and only if $z =0$; so $\overline{X}$ is also Hermitian positive definite. Now observe that $A^{-\frac{1}{2}} X A^{-\frac{1}{2}} = I + i A^{-\frac{1}{2}} B A^{-\frac{1}{2}} \succ 0$  and
\[
A + BA^{-1}B = A^{\frac{1}{2}} \bigl[ I + (A^{-\frac{1}{2}}B A^{-\frac{1}{2}}) (A^{-\frac{1}{2}} B A^{-\frac{1}{2}}) \bigr] A^{\frac{1}{2}}.
\]
Therefore it suffices to establish \ref{lem:item-2} for $A = I$. As
\[
I + iB \succ 0,\quad I - iB \succ 0, \quad I + B^2 =(I - iB )^{\frac{1}{2}} (I + iB) (I - iB )^{\frac{1}{2}},
\]
it follows that $I + B^2 \succ 0$. 
An alternative way to show \ref{lem:item-2} is to use Lemma~\ref{lem:inverse}, which informs us that $(A + BA^{-1}B)^{-1}$ is the real part of $X^{-1}$, allowing us to invoke \ref{lem:item-1}. Then $X \succ 0 \Rightarrow X^{-1} \succ 0 \Rightarrow (A + BA^{-1}B)^{-1} \succ 0 \Rightarrow A + BA^{-1}B \succ 0$.
\end{proof}

With Lemma~\ref{lem:positive real part} established, we may turn \eqref{eq:cholesky} into Algorithm~\ref{alg:inverse4}, which essentially replaces the LU decompositions in Algorithm~\ref{alg:inverse3} with Cholesky decompositions, taking care to preserve symmetry and positive definiteness.

\begin{algorithm}[!htb]
\caption{Frobenius inversion with Cholesky decomposition}
\label{alg:inverse4}
\begin{algorithmic}[1]
\Require $X = A + i B$ with $A \in \GL_n(\mathbb{R})$ 
\State Cholesky decompose $A = R_1^\tp R_1$;
\State backward substitute for $X_1$ in $R_1^\tp X_1 = B$;
\State forward substitute for $X_2$ in $R_1X_2 = X_1$;
\State matrix multiply $X_3 = X_1^\tp X_1$; \label{step:smm}
\State matrix add $X_4 = A - X_3$;
\State Cholesky decompose $X_4 = R_2^\tp R_2$;
\State backward substitute for $X_5$ in $R_2^\tp X_5 =I$;
\State forward substitute for $X_6$ in  $R_2 X_6 = X_5$;
\State matrix multiply $X_7 = X_2 X_6$;
\Ensure inverse $X^{-1} = X_6 - i X_7$
\end{algorithmic}
\end{algorithm}

The standard method for inverting a Hermitian positive definite matrix is to simply replace LU decomposition in Algorithm~\ref{alg:LU1} by Cholesky decomposition, given in Algorithm~\ref{alg:Cholesky2} for easy reference.

\begin{algorithm}[!htb]
\caption{Inversion with Cholesky decomposition}
\label{alg:Cholesky2}
\begin{algorithmic}[1]
\Require $A \in \GL_n(\mathbb{C})$
\State Cholesky decompose $A = R^\ha  R$;
\State backward substitute for $X_1$ in $R^\ha X_1 = I$;
\State forward substitute for $X_2$ in $R X_2 = X_1$;
\Ensure inverse $A^{-1} = X_2$
\end{algorithmic}
\end{algorithm}

With this, we obtain an analogue of Proposition~\ref{prop:threshold}. The flop counts below show that  Algorithm~\ref{alg:inverse4} provides a 22\%  speedup over Algorithm~\ref{alg:Cholesky2}. The experiments in Section~\ref{sec:HPDexp} will attest to this improvement.
\begin{proposition}[Flop counts]\label{prop:threshold2}
Algorithm~\ref{alg:inverse4}  has a real flop count of $23n^3/3$ whereas applying Algorithm~\ref{alg:Cholesky2} directly to a complex matrix has a real flop count of $28n^3/3$.
\end{proposition}
\begin{proof}
Algorithm~\ref{alg:Cholesky2} performs one Cholesky decomposition, $n$ backward substitutions, and $n$ forward substitutions, all over $\mathbb{C}$. So its flop complexity is dominated by $n^3/3 + n^3 + n^3 = 7n^3/3$ complex flops and thus, by the same reasoning in the proof of Theorem~\ref{thm:threshold}, $28 n^3/3$ real flops. On the other hand, Algorithm~\ref{alg:inverse4} performs two Cholesky decompositions, $2n$ backward substitutions, $2n$ forward substitutions, and two matrix multiplications, all over $\mathbb{R}$. Moreover, the symmetry in $X_1^\tp X_1$ allows the matrix multiplication in Step~\ref{step:smm} to have a reduced complexity of $n^3$ real flops. Hence its flop complexity is dominated by $ 2n^3/3 + 2n^3 + 2n^3 + 3n^3 = 23 n^3/3$ real flops.
\end{proof}

We end with an observation that the discussions in this section apply as long as $A \succ 0$ and $A + BA^{-1}B \succ 0$. Indeed, another important class of matrices with this property are the $A + iB \in \mathbb{C}^{n \times n}$ with symmetric positive definite real and imaginary parts, i.e., $A \succ 0$ and $B \succ 0$ \cite[p.~209]{Higham}. By Lemma~\ref{lem:positive real part}, such matrices are not Hermitian positive definite except in the trivial case when $B = 0$. However, since such matrices must clearly satisfy $A + BA^{-1}B \succ 0$, Algorithm~\ref{alg:inverse4} and Proposition~\ref{prop:threshold2} will apply verbatim to them.

\section{Numerical experiments}\label{sec:expr}

We present results from numerical experiments comparing the speed and accuracy of Frobenius inversion (Algorithms~\ref{alg:Frob:linear}, \ref{alg:inverse3}, \ref{alg:inverse4}) with standard methods via LU and Cholesky decompositions (Algorithms~\ref{alg:LU1}, \ref{alg:Cholesky2}). We begin by comparing Algorithms~\ref{alg:LU1} and \ref{alg:inverse3}, followed by a variety of common tasks: linear systems, matrix sign function,  Sylvester equations, Lyapunov equations, polar decomposition, and rounding up with a comparison of  Algorithms~\ref{alg:inverse4} and \ref{alg:Cholesky2} on the inversion of Hermitian positive matrices.  These results show that algorithms based on Frobenius inversion are more efficient than standard ones based on LU or Cholesky decompositions, with negligible loss in accuracy, confirming Theorem~\ref{thm:threshold}, Propositions~\ref{prop:threshold} and \ref{prop:threshold2}. In all experiments, we repeat our random trials ten times and record average time taken and average forward or backward error. All codes are available at \url{https://github.com/zhen06/Complex-Matrix-Inversion}.

\subsection{Matrix inversion}\label{sec:inv}

For our speed experiments, we generate $X = A + iB \in \mathbb{C}^{n \times n}$ with entries of $A,B \in \mathbb{R}^{n \times n}$ sampled uniformly from $[0,1]$ and  $n$ from $3600$ to  $6000$.
\begin{figure}[!h]
    \centering
    \begin{subfigure}[b]{0.49 \textwidth}
        \includegraphics[width = \textwidth]{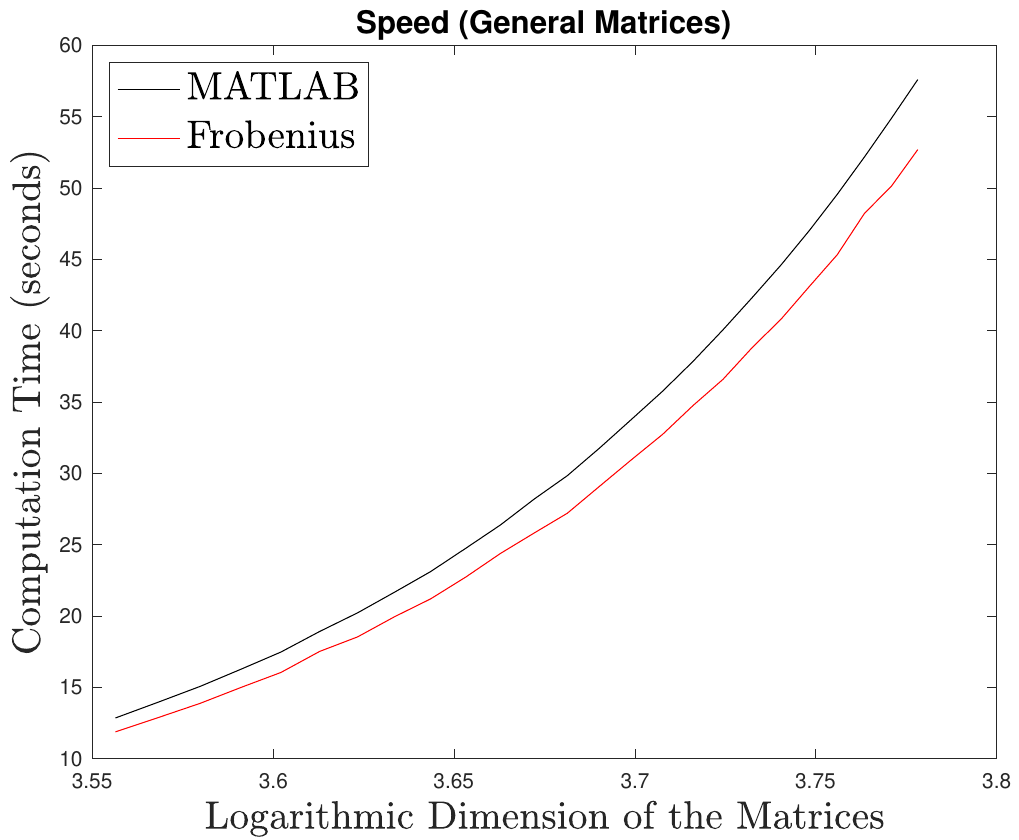}
    \end{subfigure}
    \hfill
    \begin{subfigure}[b]{0.49 \textwidth}
        \includegraphics[width = \textwidth]{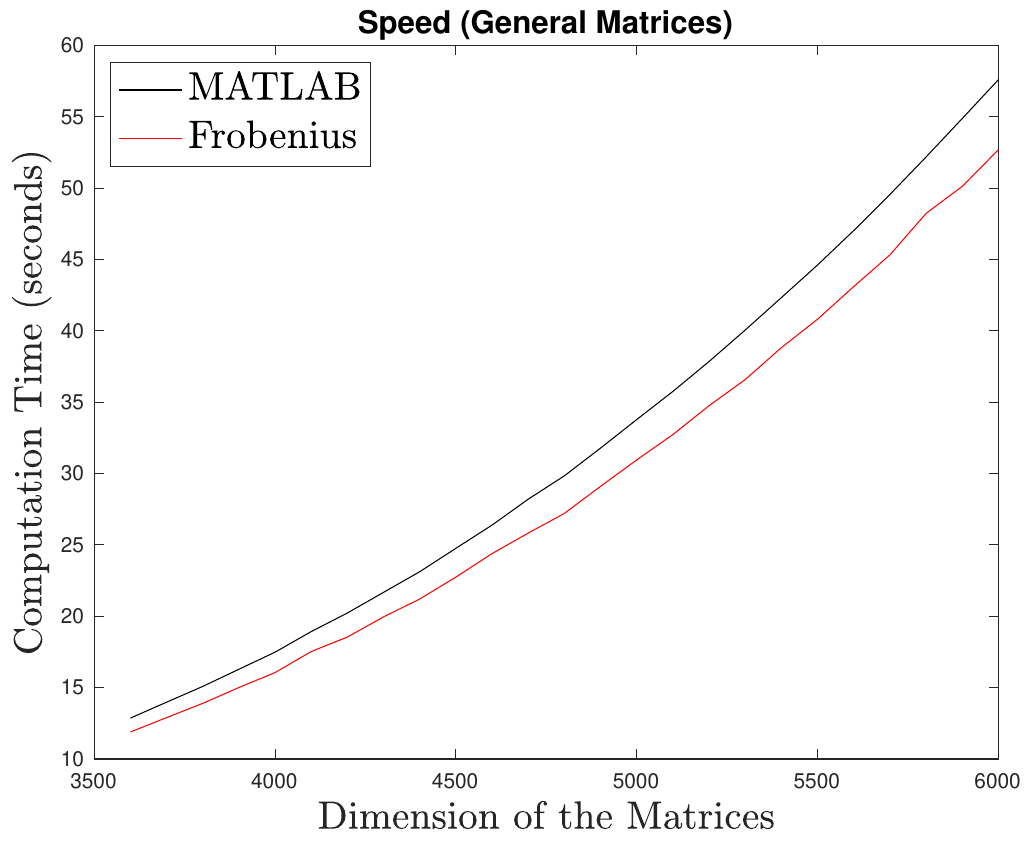}
    \end{subfigure}
    \caption{Time taken versus log-dimension (\emph{left}) and dimension (\emph{right}) of matrix.}
    \label{expr:speed}
\end{figure}

Figure~\ref{expr:speed} shows the times taken for \textsc{Matlab}'s built-in inversion (Algorithm~\ref{alg:LU1}), i.e., directly performing LU decomposition in complex arithmetic, and Frobenius inversion with LU decomposition in real arithmetic (Algorithm~\ref{alg:inverse3}). They are  plotted against matrix dimension $n$, using two different scales for the horizontal axis.  As predicted by Proposition~\ref{prop:threshold}, Frobenius inversion is indeed faster than \textsc{Matlab}'s inversion, with a widening gap as $n$ grows bigger.

For our accuracy experiments, we want some control over the condition numbers of our random matrices to reduce conditioning as a factor affecting accuracy. We generate a random $A \in \mathbb{R}^{n \times n}$ with condition number $\kappa$: first generate a random orthogonal $Q \in \O_n(\mathbb{R})$ by QR factoring a random $Y  \in \mathbb{R}^{n \times n}$ with entries sampled uniformly from $[-1,1]$; next generate a random diagonal $\Lambda = \diag(\lambda_1,\dots,\lambda_n) \in \mathbb{R}^{n \times n}$ with $\lambda_1 = \pm \kappa$, $\lambda_n = \pm 1$, signs assigned randomly, and $\lambda_2,\dots,\lambda_{n-1} \in [-\kappa, -1] \cup [1,\kappa]$  sampled uniform randomly; then set $A \coloneqq Q \Lambda Q^\tp / \lVert \Lambda \rVert_\F$. We generate $B \in \mathbb{R}^{n \times n}$ in the same way. So $\kappa(A) =\kappa(B)= \kappa$. We also check that $\kappa(X)$ is on the same order of magnitude as $\kappa$ or otherwise discard $X$. In the plots below, we set $\kappa = 10$ and increase $n$ from $2$ through  $4096$. 

Accuracy is measured by left and right \textit{relative residuals} defined respectively as
\begin{equation}\label{eq:backerr}
    \res_L(X,\widehat{Y}) \vcentcolon = \frac{\lVert \widehat{Y}X - I \rVert_{\max}}{\lVert X \rVert_{\max} \lVert \widehat{Y} \rVert_{\max}}, \qquad  \res_R(X,\widehat{Y}) \vcentcolon = \frac{\lVert X\widehat{Y} - I \rVert_{\max}}{\lVert X \rVert_{\max} \lVert \widehat{Y} \rVert_{\max}}
\end{equation}
where $\widehat{Y}$ is the computed inverse of $X$ and the \emph{max norm} is
\begin{equation} \label{eq:maxnorm}
    \lVert A + i B \rVert_{\max} \vcentcolon = \max (\lVert A \rVert_{\max}, \lVert B \rVert_{\max}) \vcentcolon = \max \Bigl(\max_{i,j =1,\dots,n} |a_{ij}|, \max_{i,j = 1,\dots,n} |b_{ij}| \Bigr).
\end{equation}
Figure~\ref{fig:acc} shows the left and right relative residuals computed by \textsc{Matlab}'s built-in inversion (Algorithm~\ref{alg:LU1}) and Frobenius inversion  (Algorithm~\ref{alg:inverse3}), plotted against matrix dimension $n$. At first glance, Frobenius inversion is less accurate than \textsc{Matlab}'s inversion. But one needs to look at the scale of the vertical axis --- the two algorithms give essentially the same results to machine precision ($15$ decimal digits of accuracy), any difference can be safely ignored for all intents and purposes.
\begin{figure}[!htb]
    \centering
    \begin{subfigure}[b]{0.49 \textwidth}
        \includegraphics[width = \textwidth]{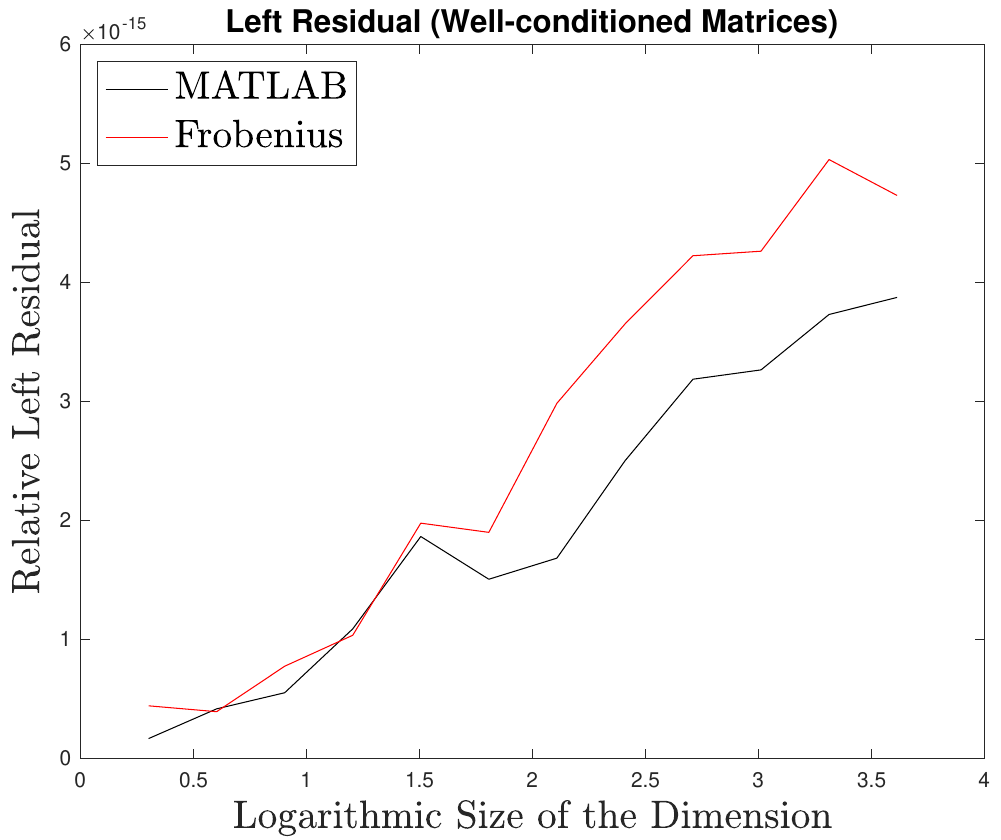}
    \end{subfigure}
    \hfill
    \begin{subfigure}[b]{0.49 \textwidth}
        \includegraphics[width = \textwidth]{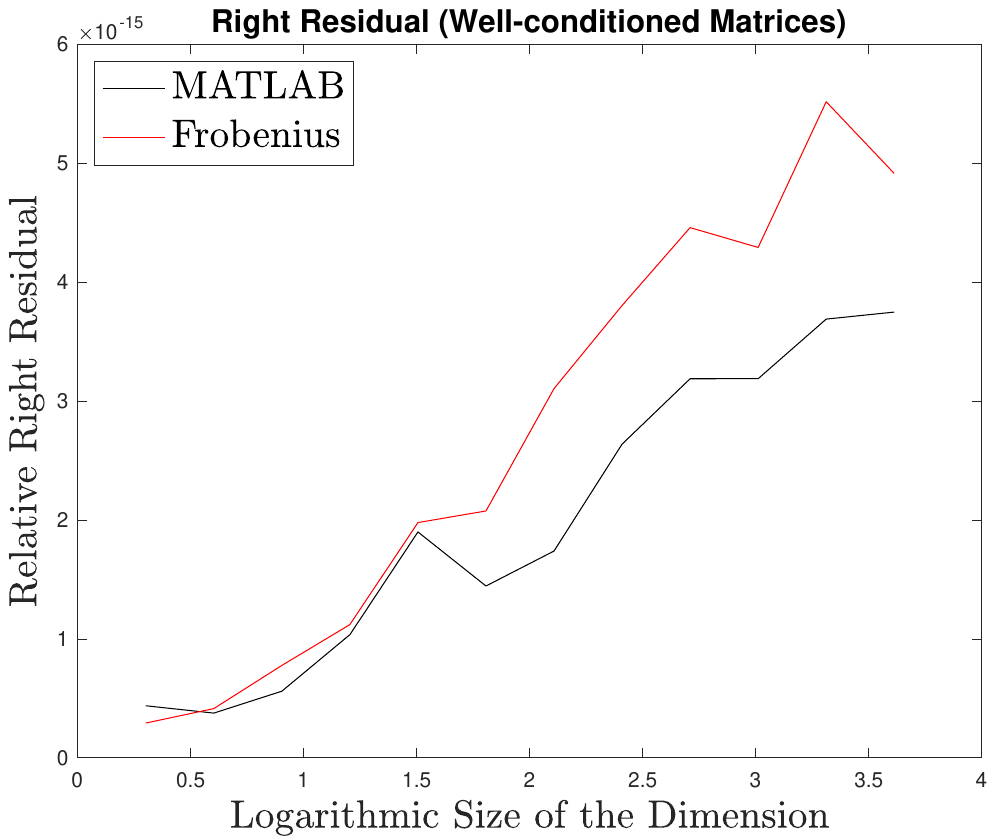}
    \end{subfigure}
    \caption{Relative left and right residuals of Frobenius inversion versus \textsc{Matlab} built-in inversion. Note that scale of the vertical axis is $10^{-15}$.}
    \label{fig:acc}
\end{figure}

\subsection{Solving linear systems}\label{sec:linexp}

It is remarkable that Frobenius inversion shows nearly no loss in accuracy as measured by backward error. For the matrix dimensions in Section~\ref{sec:inv}, forward error experiments are too expensive due to the cost of finding exact inverse. To get a sense of the forward errors, we look at a problem intimately related to matrix inversion --- solving linear systems.

We use the same matrices generated in Section~\ref{sec:inv} and generate two vectors $x,y\in \mathbb{R}^n$ with entries sampled uniformly from $[-1,1]$. We set $c+id \coloneqq (A+iB)(x+iy)$ and solve the complex linear system $(A+iB)(x+iy) = c+id$ to get a computed solution $\widehat{x} + i\widehat{y}$ using three methods: (i) Frobenius inversion (Algorithm~\ref{alg:Frob:linear}), (ii) complex LU factorization, and (iii) augmented system $\begin{bsmallmatrix} A & -B \\ B & A \end{bsmallmatrix} \begin{bsmallmatrix} x \\ y \end{bsmallmatrix} = \begin{bsmallmatrix} c \\ d \end{bsmallmatrix}$; we rely on \textsc{Matlab}'s \texttt{mldivide} (i.e., the `backslash' operator) for the last two.

\begin{figure}[!htb]
    \centering
    \includegraphics[scale = 0.5]{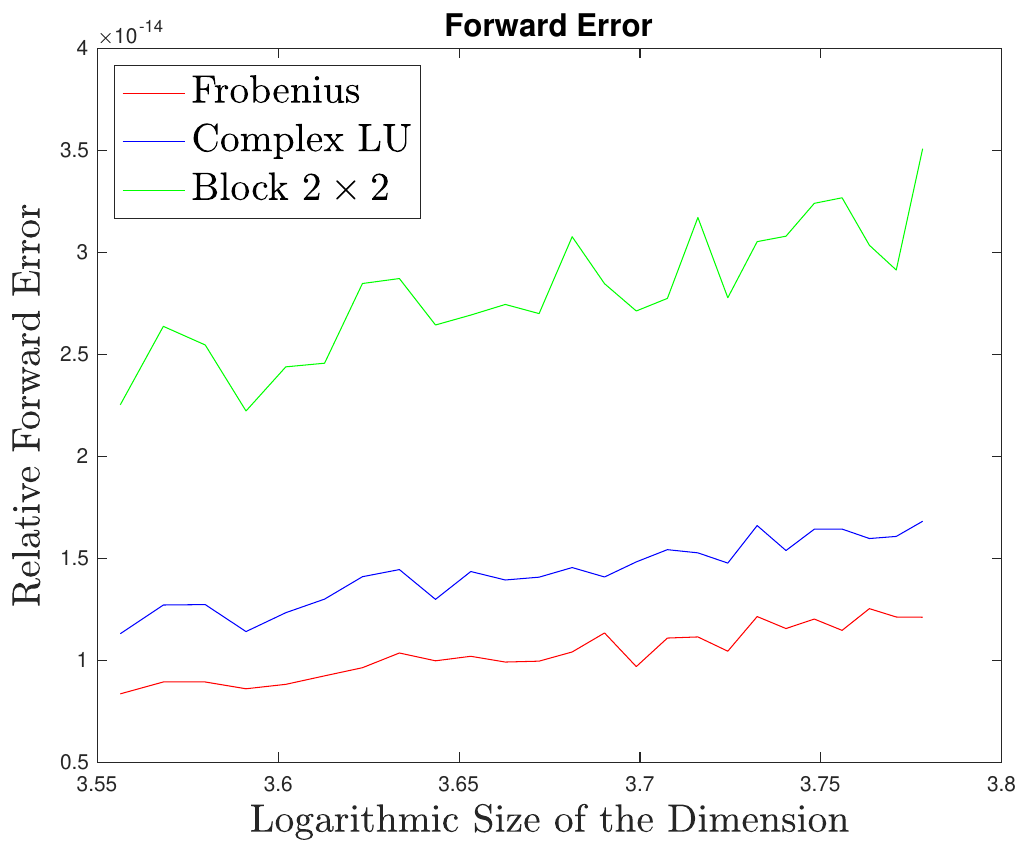}
    \caption{Linear systems with with Frobenius inversion and \textsc{Matlab}'s backslash.}
    \label{fig:linear}
\end{figure}

Figure~\ref{fig:linear} shows the relative forward errors  $\lVert\widehat{x} + i \widehat{y} - (x  + iy) \rVert_{\max} / \lVert x + i y \rVert_{\max}$ plotted against  matrix dimension. The conclusion is clear: Frobenius inversion gives the most accurate result.

\subsection{Matrix sign function}\label{sec:sign}

The matrix sign function appears in a wide range of problems such as algebraic Riccati equation \cite{sign_original}, Sylvester equation \cite{polar_sylvester_higham,sign_original}, polar decomposition \cite{polar_sylvester_higham}, and spectral decomposition \cite{eigen_ref1,eigen_ref3,eigen_ref2,eigen_ref4,eigen_ref5}. 
For $X \in \GL_n(\mathbb{C})$ with Jordan decomposition $X = ZJZ^{-1}$ where its Jordan canonical form $J = \begin{bsmallmatrix}
J_{\p} & 0 \\
0 & J_{\m}
\end{bsmallmatrix}$  is partitioned into $J_{\p} \in \mathbb{C}^{p \times p}$ with positive real part and $J_{\m} \in \mathbb{C}^{q \times q}$ with negative  real part, its matrix sign function is defined to be
\begin{equation}\label{eq:sign}
    \sign(X) = Z 
    \begin{bmatrix}
    I_p & 0 \\
    0 & -I_q
    \end{bmatrix}Z^{-1}.
\end{equation}
Since the Jordan decomposition cannot be determined in finite precision \cite{GW}, its definition does not offer a viable way of computation. The standard way to evaluate the matrix sign function is to use Newton iterations \cite{functions_of_matrices_higham,sign_original}:
\begin{equation}\label{eq:iter}
    X_{k+1} = \frac{1}{2}(X_k + X_k^{-1}), \quad k =0,1,2,\dots, \quad X_0 = X.
\end{equation}
This affords a particularly pertinent test for Frobenius inversion as it involves repeated inversion. Our stopping condition is given by the relative change in $X_k$: We stop when $ \lVert X_k - X_{k-1} \rVert_{\max} / \lVert X_k \rVert_{\max} \leq \varepsilon = 10^{-3}$ or when $k \ge k_{\max} = 100$.

The definition via Jordan decomposition  is useful for generating random examples for our tests: We generate a random diagonal $J \in \mathbb{C}^{n \times n}$ whose first $p \approx n/2$ diagonal entries have positive real parts and the rest have negative real parts, avoiding near zero values, and with $n$ from $2100$ to $4000$. We generate a random $Z \in \GL_n(\mathbb{C})$ with real and imaginary parts of its entries $z_{ij}$ sampled uniformly from $[-1,1]$. We set $X \coloneqq Z J Z^{-1}$. In this way we obtain $\sign(X)$ via \eqref{eq:sign} as well.
 
In each iteration of \eqref{eq:iter}, we compute $X_k^{-1}$ with \textsc{Matlab}'s inversion in complex arithmetic (Algorithm~\ref{alg:LU1})  and Frobenius inversion in real arithmetic (Algorithm~\ref{alg:inverse3}). Accuracy is measured by relative forward error $\lVert \sign(X) - \widehat{S} \rVert_{\max}/\lVert \sign(X) \rVert_{\max}$. 
From Figure~\ref{fig:sign}, we see that Frobenius inversion offers an improvement in speed at the cost of slightly less accurate results.
\begin{figure}[!htb]
    \centering
    \begin{subfigure}[b]{0.49 \textwidth}
        \includegraphics[width = \textwidth]{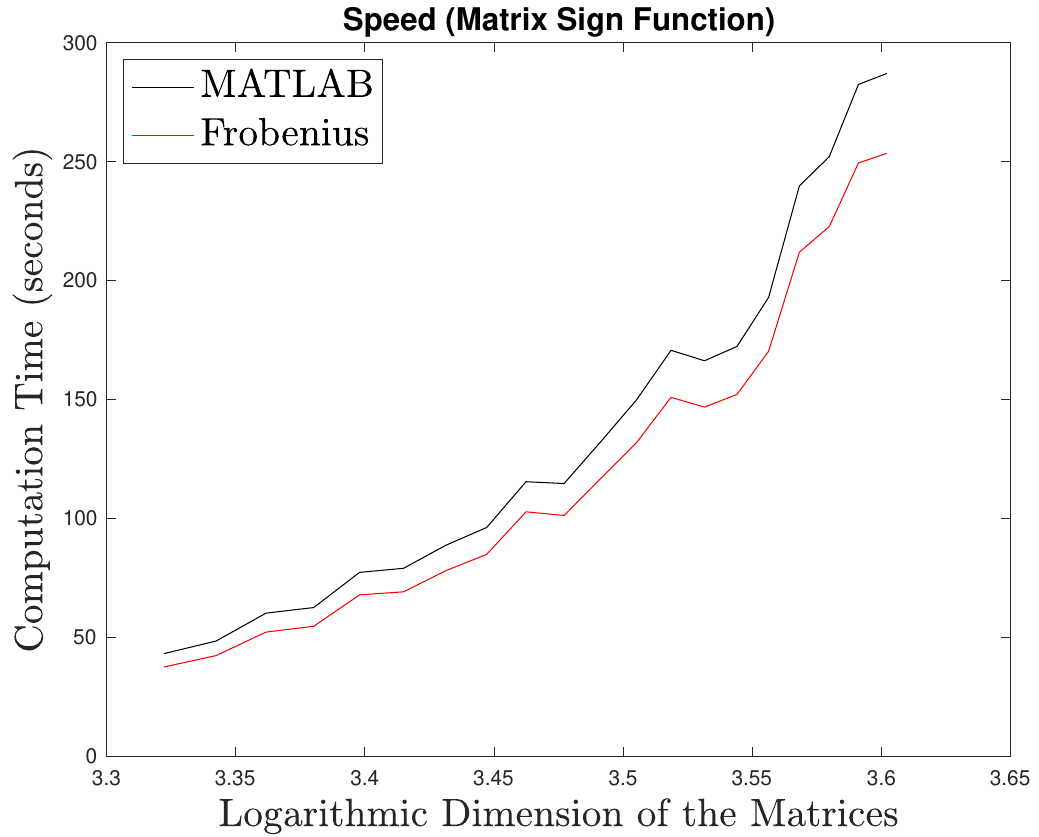}
    \end{subfigure}
    \hfill
    \begin{subfigure}[b]{0.49 \textwidth}
        \includegraphics[width = \textwidth]{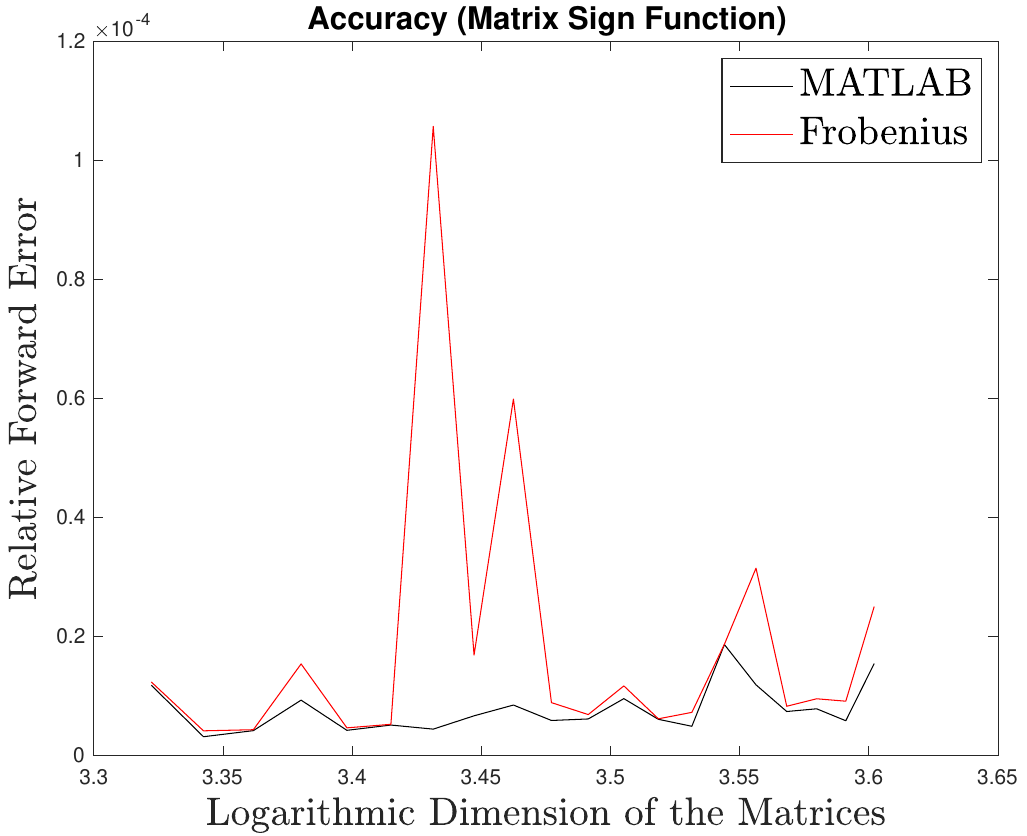}
    \end{subfigure}
    \caption{Matrix sign function with Frobenius inversion and \textsc{Matlab}'s inversion.}
    \label{fig:sign}
\end{figure}
 
\subsection{Sylvester and Lyapunov equations}

One application of the matrix sign function is to seek, for given $A \in \mathbb{C}^{p \times p}$, $B \in \mathbb{C}^{q\times q}$, and $C \in \mathbb{C}^{p \times q}$, a solution $Y \in \mathbb{C}^{p \times q}$ for the \emph{Sylvester equation}:
\[
    AY + YB = C,
\]
or its special case with $B = A^\ha$, the \emph{Lyapunov equation} \cite{HighamBook}. As noted in \cite{polar_sylvester_higham,sign_original}, if $\sign(A) = I_p$ and $\sign(B) = I_q$, then
\[
    \sign \left( \begin{bmatrix}
    A & -C \\
    0 & -B
    \end{bmatrix} \right)
    =
    \begin{bmatrix}
    I_p & -2Y \\
    0 & -I_q
    \end{bmatrix}.
\]
Thus the Newton iterations \eqref{eq:iter} applied to $ X_0 = \begin{bsmallmatrix}
    A & -C \\
    0 & -B
    \end{bsmallmatrix}$  will converge to
$\begin{bsmallmatrix}
    I & -2Y \\
    0 & -I
\end{bsmallmatrix}$, yielding the solution $Y$ of Sylvester equation in the limit.

As usual, we `work backwards' to generate  $A \in \mathbb{C}^{p \times p}$ with  $\sign(A) = I_p$, $B \in \mathbb{C}^{q\times q}$ with  $\sign(B) = I_q$, and $C \in \mathbb{C}^{p \times q}$ with $p$ and $q$ taking values between $1050$ and $2000$.  First we generate a random $Z \in \GL_p(\mathbb{C})$ by sampling the real and imaginary parts of its entries in $[-1,1]$ uniformly; next we generate a random diagonal $J \in \mathbb{C}^{n \times n}$ whose diagonal entries have positive real parts sampled from the interval $[9,10]$; then we set $A \coloneqq Z J Z^{-1} \in \mathbb{C}^{p \times p}$. We generate $B \in \mathbb{C}^{q \times q}$ in the same way. We  generate a random $Y \in \mathbb{C}^{p \times q}$ with real and imaginary parts of its entries sampled uniformly from $[-1,1]$ and set $C \coloneqq A Y + Y B$. 

Using the same stopping condition in Section~\ref{sec:sign} with a tolerance of $\varepsilon = 10^{-1}$ and $k_{\max} = 100$ maximum iterations, we compute a solution $\widehat{Y}$ with the Newton iterations  \eqref{eq:iter}. Accuracy is measured by relative forward error  $\lVert Y - \widehat{Y} \rVert_{\max}/\lVert Y \rVert_{\max}$.

Figure~\ref{fig:syl} gives the results for Sylvester and Lyapunov equations, showing that in both cases Frobenius inversion is faster than \textsc{Matlab}'s inversion with no difference in accuracy.  Indeed, at a scale of $10^{-5}$ for the vertical axis, the two graphs in the accuracy plot for Lyapunov equation (bottom right plot of Figure~\ref{fig:syl}) are indistinguishable. The  accuracy plot for Sylvester equation (top right plot of Figure~\ref{fig:syl}) uses a finer vertical scale of $10^{-8}$; but had we used a scale of $10^{-5}$, the two graphs therein would also have been indistinguishable.

\begin{figure}[!htb]
    \centering
    \begin{subfigure}[b]{0.49 \textwidth}
        \includegraphics[width = \textwidth]{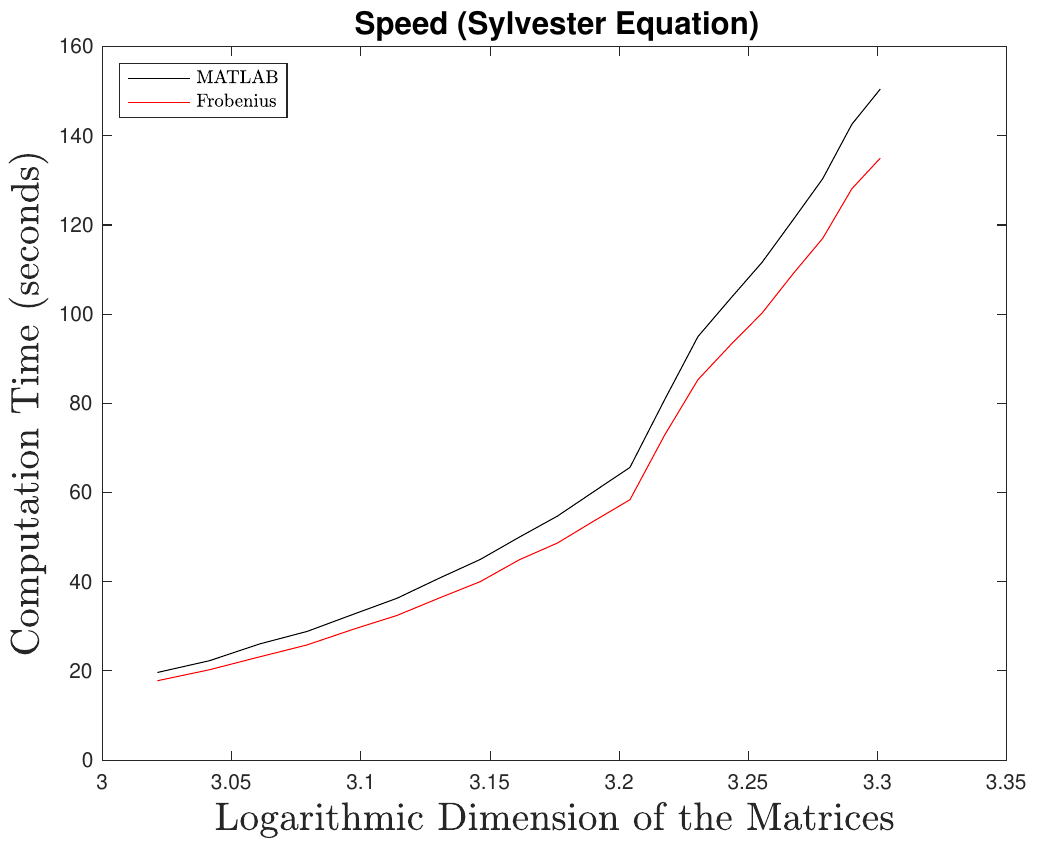}
    \end{subfigure}
    \hfill
    \begin{subfigure}[b]{0.49 \textwidth}
        \includegraphics[width = \textwidth]{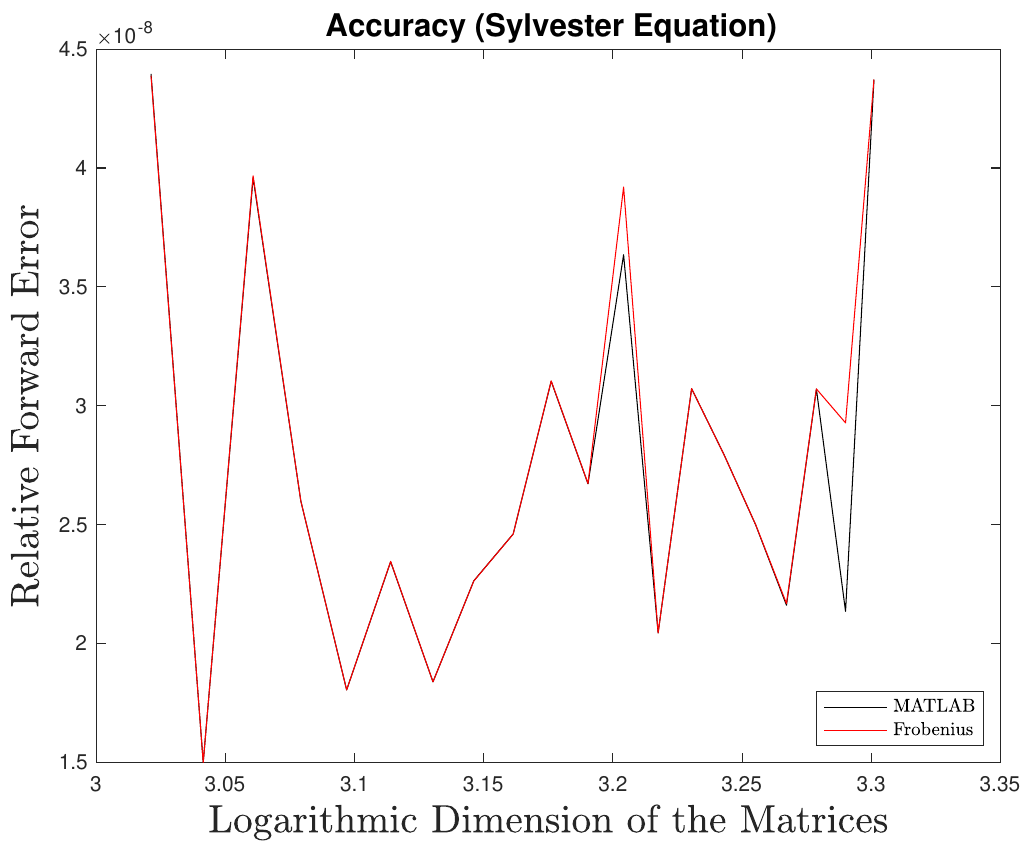}
    \end{subfigure}
    \begin{subfigure}[b]{0.49 \textwidth}
        \includegraphics[width = \textwidth]{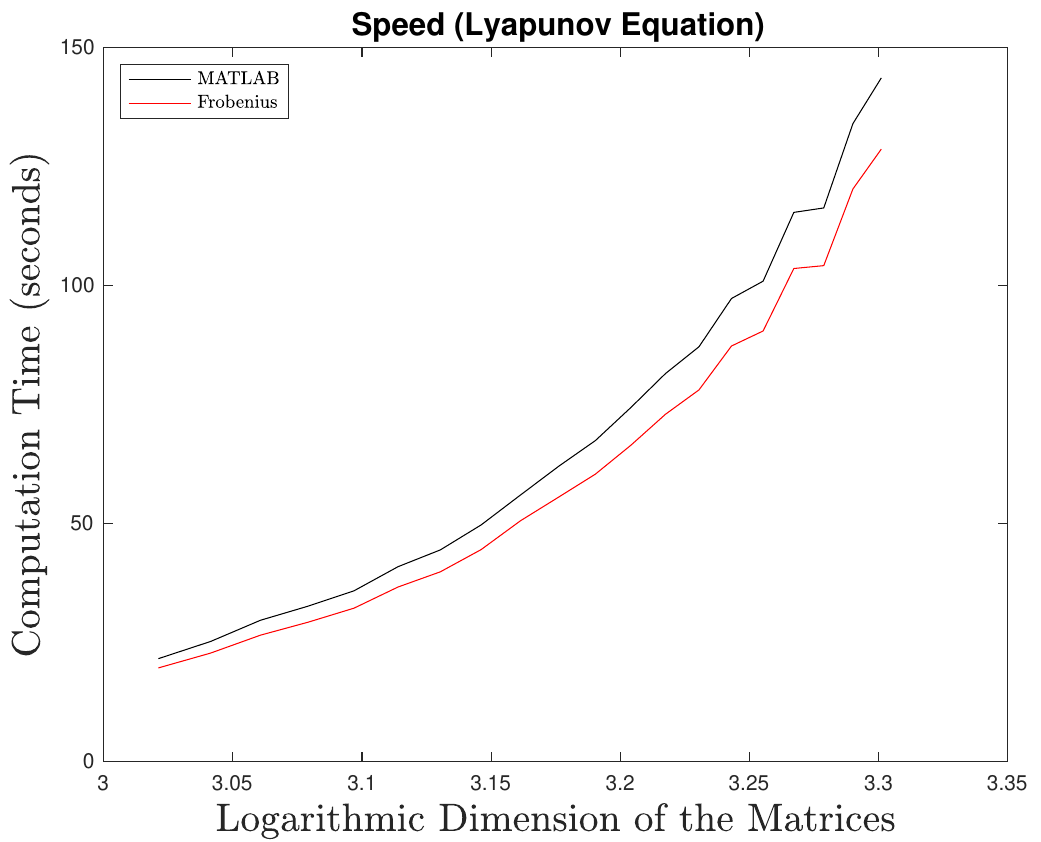}
    \end{subfigure}
    \hfill
    \begin{subfigure}[b]{0.49 \textwidth}
        \includegraphics[width = \textwidth]{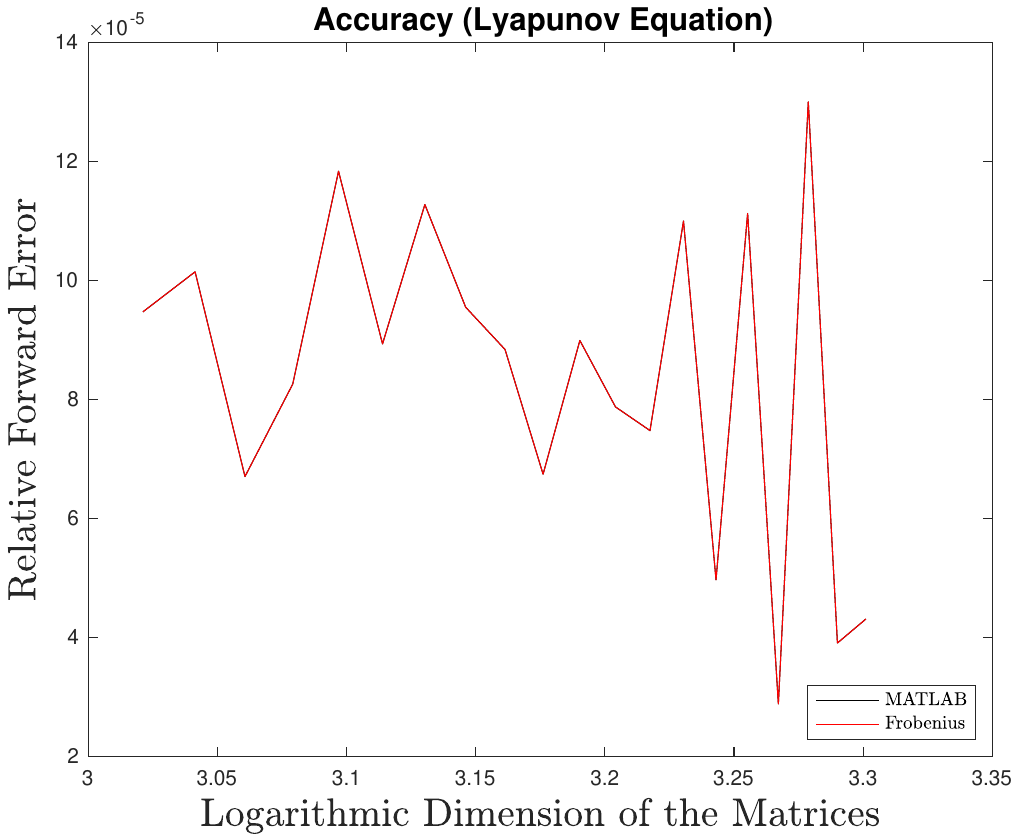}
    \end{subfigure}
    \caption{Sylvester (\emph{top}) and Lyapunov  (\emph{bottom})  equations with Frobenius inversion and \textsc{Matlab}'s inversion. Note there are two graphs in the bottom right plot.}
    \label{fig:syl}
\end{figure}

\subsection{Polar decomposition}

Another application of the matrix sign function is to polar decompose a given $X \in \mathbb{C}^{n \times n}$ into $X = Q P$ with $Q \in \U_n(\mathbb{C})$ and $P \in \mathbb{C}^{n \times n}$  Hermitian positive semidefinite. This is based on the observation \cite{polar_ref2,polar_sylvester_higham,polar_ref3} that
\[
    \sign \biggl(
    \begin{bmatrix}
    0 & X \\
    X^\ha  & 0
    \end{bmatrix} \biggr)
    = \begin{bmatrix}
    0 & Q \\
    Q^\ha  & 0
    \end{bmatrix}.
\]
Here the Newton iterations \eqref{eq:iter} take a slightly different form
\begin{equation}\label{eq:iter2}
    X_{k+1} = \frac{1}{2}(X_k + X_k^{-\ha}), \quad k =0,1,2,\dots, \quad X_0 = X.
\end{equation}

We generate random $Y, Z \in \mathbb{C}^{n \times n}$ with real and imaginary parts of its entries sampled uniformly from $[-1,1]$. We  then QR factorize $Y = UR$ and l set $P \coloneqq Z^\ha  Z$ and $X = UP$. The value of $n$ runs from $2100$ to $4000$.

Using the same stopping condition in Section~\ref{sec:sign} with a tolerance of $\varepsilon = 10^{-3}$ and $k_{\max} = 100$ maximum iterations, we compute a solution $\widehat{Q}$ with the Newton iterations \eqref{eq:iter2}, with $X_k^{-\ha}$ computed by either Frobenius inversion or \textsc{Matlab}'s inversion. We then set $\widehat{P} = \widehat{Q}^\ha  X$.  Accuracy is measured by relative forward errors  $\lVert Q - \widehat{Q} \rVert_{\max}/\lVert Q \rVert_{\max}$ and $\lVert P - \widehat{P} \rVert_{\max}/\lVert P \rVert_{\max}$.

\begin{figure}[!htb]
    \centering
    \begin{subfigure}[b]{0.49 \textwidth}
        \includegraphics[width = \textwidth]{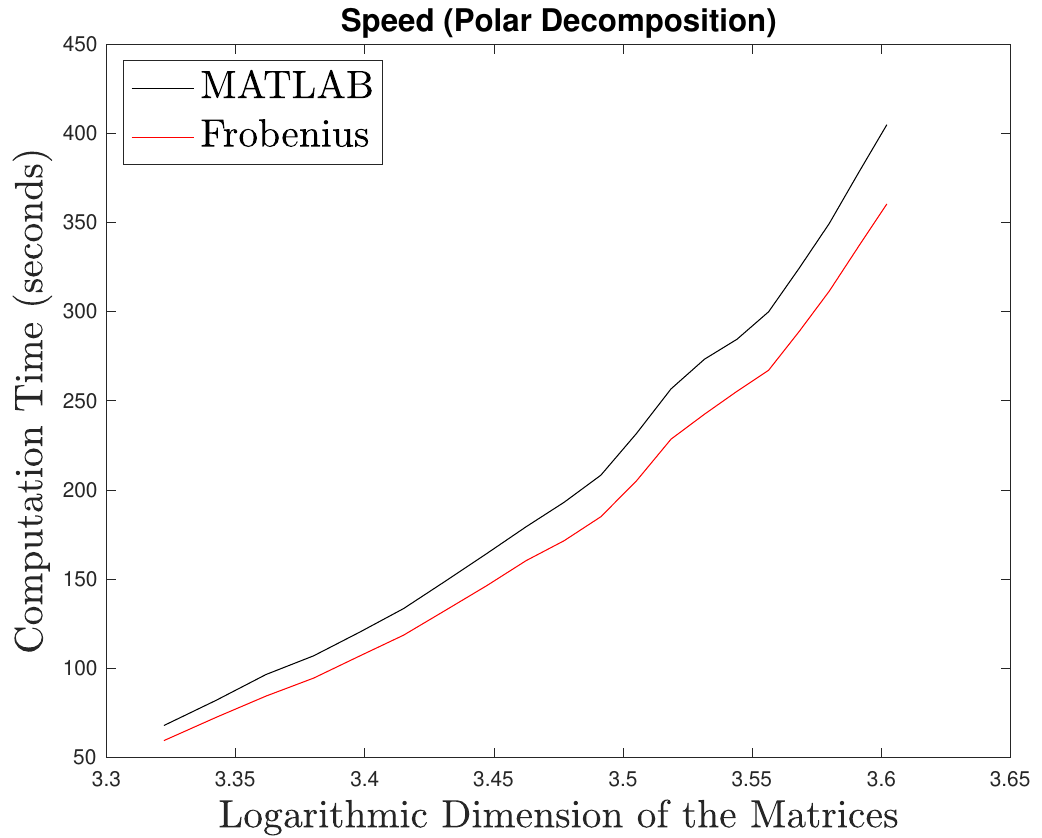}
    \end{subfigure}
    \hfill
    \begin{subfigure}[b]{0.49 \textwidth}
        \includegraphics[width = \textwidth]{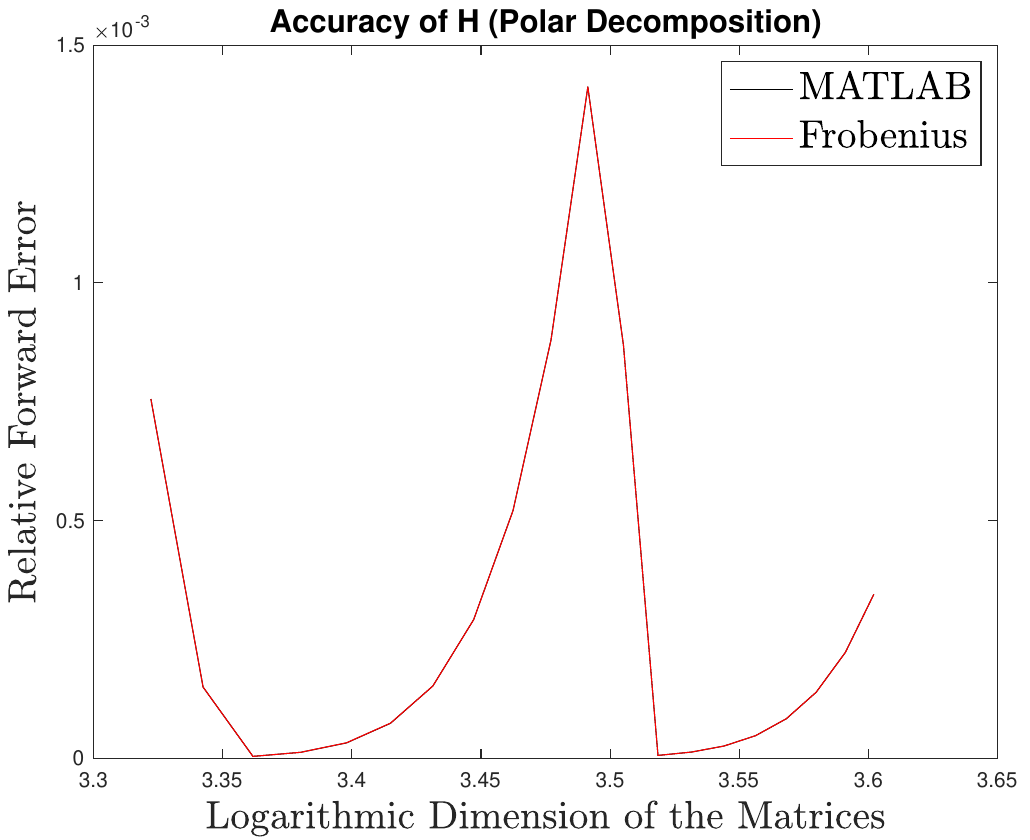}
    \end{subfigure}
    \caption{Polar decomposition with Frobenius inversion and \textsc{Matlab}'s inversion. Note that there are two graphs in the right plot}
    \label{fig:polar}
\end{figure}

Figure~\ref{fig:polar} again shows that Frobenius inversion is faster than \textsc{Matlab}'s built-in inversion with near-identical accuracy. Indeed, at a scale of $10^{-3}$ for the vertical axis, the two graphs in the accuracy plot (right plot of Figure~\ref{fig:polar}) are indistinguishable.

\subsection{Hermitian positive definite matrix inversion}\label{sec:HPDexp}

We repeat experiments in Section~\ref{sec:inv} on Hermitian positive definite matrices for our variant of Frobenius inversion (Algorithm~\ref{alg:inverse4}) and \textsc{Matlab}'s built-in inversion based on Cholesky decomposition (Algorithm~\ref{alg:Cholesky2}). For comparison, we also include Algorithms~\ref{alg:LU1} and \ref{alg:inverse3} that do not exploit Hermitian positive definiteness.

For our speed experiments, we generate a random Hermitian positive definite  $X \coloneqq (A + i B)^\ha(A + i B)  + 0.01 I \in \mathbb{C}^{n \times n}$ with $A,B \in \mathbb{R}^{n \times n}$ sampled uniformly from $[-1,1]$ and $n$ from $3600$ to $6000$. We plot the results in Figure~\ref{pos:speed}, with two different scales for the horizontal axis.
\begin{figure}[!htb]
    \centering
    \begin{subfigure}[b]{0.49 \textwidth}
        \includegraphics[width = \textwidth]{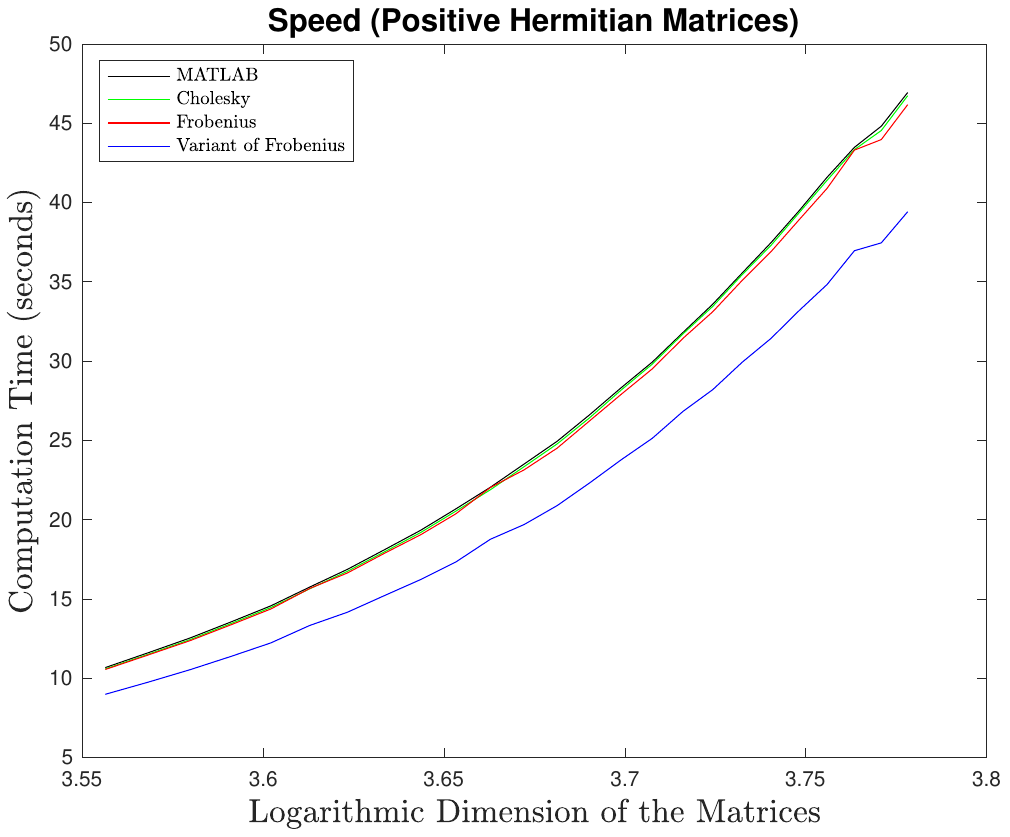}
    \end{subfigure}
    \hfill
    \begin{subfigure}[b]{0.49 \textwidth}
        \includegraphics[width = \textwidth]{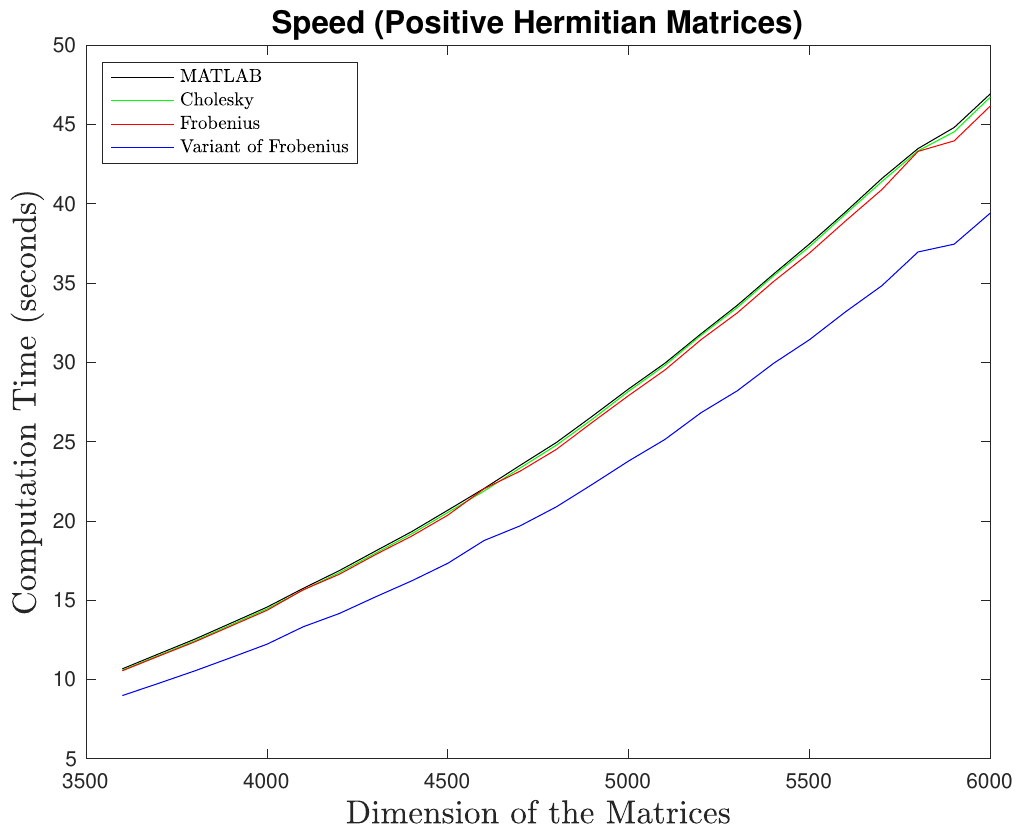}
    \end{subfigure}
    \caption{Time taken versus log-dimension (\emph{left}) and dimension (\emph{right}) of matrix.}
    \label{pos:speed}
\end{figure}

For our accuracy experiments, we control the condition numbers of our  matrices to reduce conditioning as a factor affecting accuracy. To generate a random Hermitian positive definite $X \in \mathbb{C}^{n \times n}$ with condition number $\kappa$, first we generate a random unitary $Q \in \U_n(\mathbb{C})$ by QR factoring a random $Y  \in \mathbb{C}^{n \times n}$ with real and imaginary parts of its entries sampled uniformly from $[-1,1]$; next we generate a random diagonal $\Lambda = \diag(\lambda_1,\dots,\lambda_n) \in \mathbb{R}^{n \times n}$ with $\lambda_1 = \kappa $, $\lambda_n = 1$, and $\lambda_2,\dots,\lambda_{n-1} \in [1,\kappa]$  sampled uniform randomly; then we set $X \coloneqq Q \Lambda Q^\ha / \lVert \Lambda \rVert_\F$.  So $\kappa(X) = \kappa$. In the plots below, we set $\kappa = 10$ and increase $n$ from $2$ through  $4096$. 

\begin{figure}[!htb]
    \centering
    \begin{subfigure}[b]{0.49 \textwidth}
        \includegraphics[width = \textwidth]{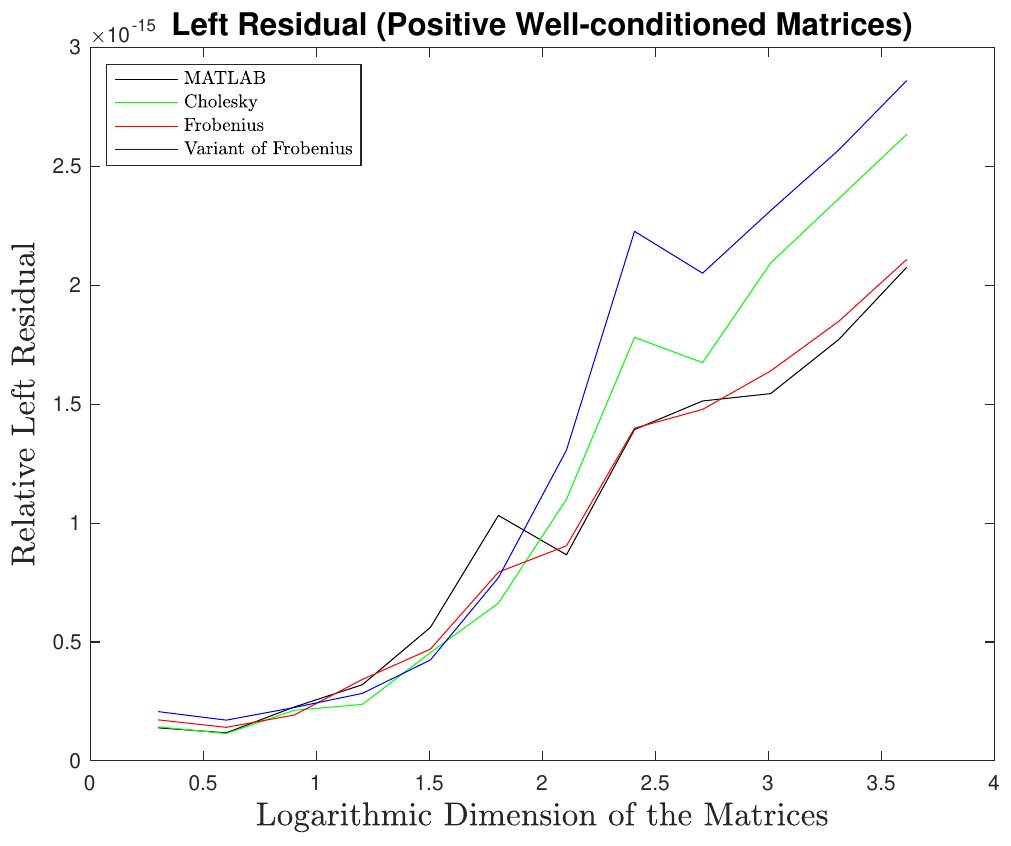}
    \end{subfigure}
    \hfill
    \begin{subfigure}[b]{0.49 \textwidth}
        \includegraphics[width = \textwidth]{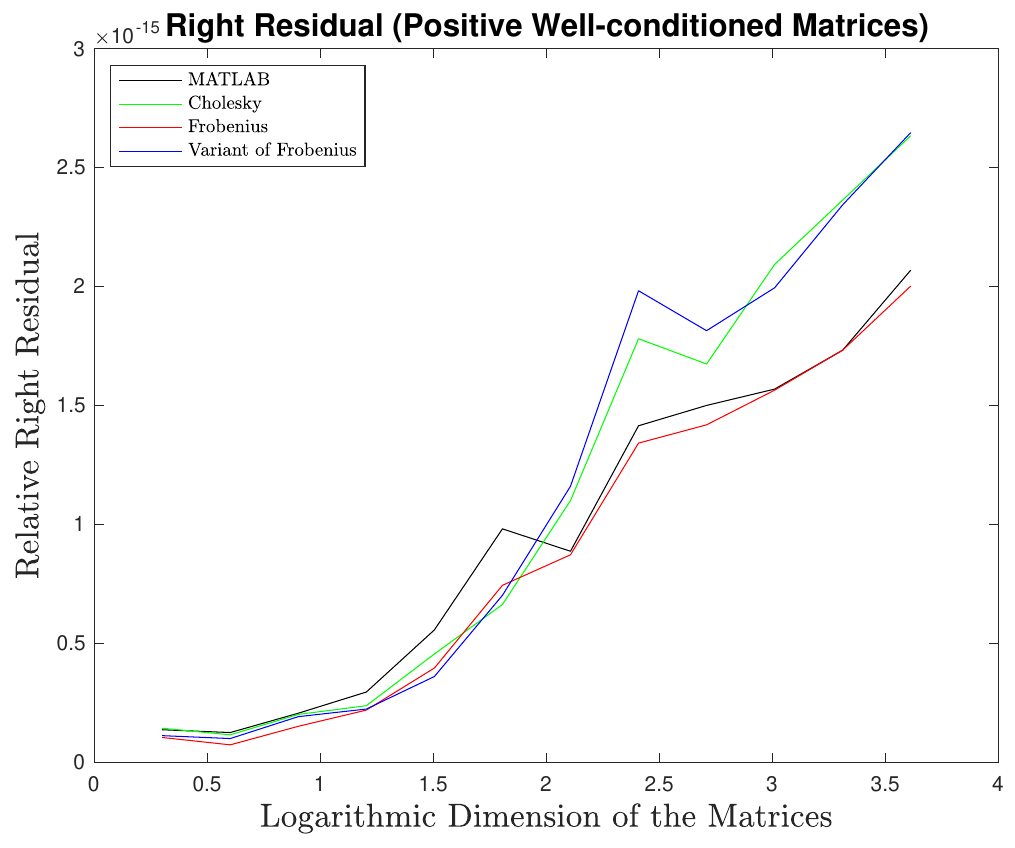}
    \end{subfigure}
    \caption{Relative left and right residuals of Algorithms~\ref{alg:LU1}, \ref{alg:inverse3}, \ref{alg:inverse4}, \ref{alg:Cholesky2}. Note that scale of the vertical axis is $10^{-15}$.}
    \label{pos:acc}
\end{figure}

Accuracy is measured by left and right relative residuals as defined in equation~\eqref{eq:backerr}, with results plotted in Figure~\ref{pos:acc}, which shows the left and right relative residuals computed by Algorithms~\ref{alg:LU1}, \ref{alg:inverse3}, \ref{alg:inverse4}, \ref{alg:Cholesky2}  plotted against matrix dimension $n$. The important thing to note is the scale of the vertical axes --- all four algorithms give essentially the same results up to machine precision.

\section{Conclusion}

We hope our effort here will rekindle interest in this beautiful algorithm. In future work, we plan to provide rounding error analysis for Frobenius inversion, discuss  its relation with Strassen-style algorithms, and its advantage in solving linear systems with a large number of right-hand sides.


\subsection*{Acknowledgment}

ZD acknowledges the support of DARPA HR00112190040 and NSF ECCF 2216912. LHL acknowledges the support of DARPA HR00112190040, NSF DMS 1854831, and a Vannevar Bush Faculty Fellowship ONR  N000142312863. KY acknowledges the support of CAS Project for Young Scientists in Basic Research, Grant No.~YSBR-008, National Key Research and Development Project No.~2020YFA0712300, and National Natural Science Foundation of China Grant No.~12288201.

\bibliographystyle{abbrv}
\bibliography{inversion-ref}
\end{document}